\documentclass[12pt]{article}
\usepackage{graphicx}
\usepackage{amsmath}
\usepackage{amsthm}
\usepackage{latexsym}
\usepackage{amsfonts}
\usepackage{etoolbox}
\usepackage{lipsum}
\let\bbordermatrix\bordermatrix
\patchcmd{\bbordermatrix}{8.75}{4.75}{}{}
\patchcmd{\bbordermatrix}{\left(}{\left[}{}{}
\patchcmd{\bbordermatrix}{\right)}{\right]}{}{}
\usepackage{multirow}
\usepackage{lipsum}
\newcommand\blfootnote[1]{%
  \begingroup
  \renewcommand\thefootnote{}\footnote{#1}%
  \addtocounter{footnote}{-1}%
  \endgroup
}
\newtheorem{theorem}{Theorem}[section]
\newtheorem{lemma}[theorem]{Lemma}
\newtheorem{corollary}[theorem]{Corollary}

\theoremstyle{definition}

\usepackage[T1]{fontenc}
\usepackage[utf8]{inputenc}
\usepackage{authblk}

\title {\bf On energy ordering of vertex-disjoint bicyclic sidigraphs}
\author {Sumaira Hafeez\thanks{Corresponding author.} }
\author {Rashid Farooq}
\affil{School of Natural Sciences,
	National University of Sciences and Technology,\\
	H-12 Islamabad, Pakistan }

\date{}
\begin{document}
\maketitle
\date{}
\blfootnote{\raggedright Email addresses:  sumaira.hafeez123@gmail.com (S. Hafeez), farook.ra@gmail.com (R. Farooq).}
\begin{abstract}
The energy and iota energy of signed digraphs are respectively defined by $E(S)=$\\
$\sum_{k=1}^n|{\rm Re}(\rho_k)|$ and $E_c(S)=\sum_{k=1}^n|{\rm Im }(\rho_k)|$, where $\rho_1, \dots,\rho_n$ are eigenvalues of $S$ and, ${\rm Re}(\rho_k)$ and ${\rm Im}(\rho_k)$ are respectively real and imaginary values of the eigenvalue $\rho_k$. Recently, Yang and Wang (2018) find the energy and iota energy ordering of digraphs in $\mathcal{D}_n$ and compute the maximal energy and iota energy, where  $\mathcal{D}_n$ denotes the set of vertex-disjoint bicyclic digraphs of a fixed order $n$. In this paper, we investigate the energy ordering of signed digraphs in $\mathcal{D}_n^s$ and finds the maximal energy, where  $\mathcal{D}_n^s$ denotes the set of vertex-disjoint bicyclic sidigraphs of a fixed order $n$. 
\end{abstract}
\begin{quote}
	{\bf Keywords:}
	Signed digraphs; Energy ordering; Maximal energy
\end{quote}

\begin{quote}
	{\bf AMS Classification:} 05C35, 05C50
\end{quote}
\section{Introduction}
If every arc of a digraph is assigned a weight $+1$ or $-1$ then it is called a signed digraph (henceforth, sidigraph). Each arc of a sidigraph is called a signed arc. We denote by $uw$, the arc from a vertex $u$ to a vertex $w$. The sign of the arc $uw$ is denoted by $\varphi(u, w)$. A directed signed path $P_n$ is a sidigraph on $n$ vertices $\{w_j\mid ~j=1,2,\dots,n\}$ with signed arcs $\{w_jw_{j+1}\mid~ j = 1, 2,\dots, n - 1\}$. A signed directed cycle $C_n$ of order $n\geq 2$ is a sidigraph with vertices $\{w_j\mid~ j=1,2,\dots,n\}$ and signed arcs $\{w_jw_{j+1}\mid~j=1,2,\dots,n-1\} \cup \{w_nw_1\}$. The product of sign of the arcs of a sidigraph $S$ is called the sign of $S$. A sidigraphs $S$ is said to be a strongly connected sidigraph if for every pair of vertices $v,w$, a path from $v$ to $w$ and a path from $w$ to $v$ exist.

A sidigraph with  equal number of vertices and arcs and contains only one directed cycle is said to be a unicyclic sidigraph. A sidigraph with connected underlying sigraph and has exactly two directed cycles is said to be bicyclic sidigraph. We denote a positive (respectively, negative) cycle of order $n$ by $C_n$ (respectively, ${\boldsymbol{C}}_n$). A cycle of order $n$ which is either positive or negative is denoted by $\mathcal{C}_n$. A positive cycle is a cycle with positive sign and a negative cycle is a cycle with negative sign. We denote by $\mathcal{D}_n^s$, the class of vertex-disjoint bicyclic sidigraphs of a fixed order $n$. 

Let $S$ be an $n$-vertex sidigraphs. Then the adjacency matrix $A(S)= [a_{ij}]_{n\times n }$ of $S$ is given by:
\[
a_{ij} =
\left\{
\begin{array}{ll}
\varphi(w_i, w_j) &\mbox{if there is an arc from $w_i$ to $w_j$,}\\
0 &\mbox{otherwise.}
\end{array}
\right.
\] 
The eigenvalues of $A(S)$ are said to be the eigevalues of $S$.

 Pe$\tilde{{\rm n}}$a and Rada \cite{IPJ} put forward the idea of digraph energy. Let $\rho_1, \dots,\rho_n$ are the eigenvalues of a sidigraph $S$. Pirzada and Bhat \cite{PB2014} defined the energy of a sidigraph $S$ as $E(S)=\sum_{k=1}^n|{\rm Re}(\rho_k)|$, where ${\rm Re}(\rho_k)$ denotes the real value of the eigenvalue $\rho_k$. Khan et al. \cite{KFR2016} and Farooq et al. \cite{FKC} put forward the idea of iota energy of digraph (sidigraph) and defined iota energy as $E_c(S)=\sum_{k=1}^n|{\rm Im}(\rho_k)|$, where ${\rm Im}(\rho_k)$ denotes the imaginary value of the eigenvalue $\rho_k$.  Khan et al. \cite{KF, KFAA} finds the extremal energy of digraphs and sidigraphs among all vertex-disjoint bicyclic digraphs and sidigraphs of order $n$. Farooq et al. \cite{FMA2017, FCK} finds the extremal iota energy of digraphs and sidigraphs among all vertex-disjoint bicyclic digraphs and sidigraphs of order $n$. In 2016, Monslave and Rada \cite{MJR} investigate the general class of bicyclic digraphs and finds extremal energy. Hafeez et al. \cite{SFK} considered the class of all bicyclic sidigraphs and finds extremal energy.

Recently, Yang and Wang \cite{XL2019} determined the energy and iota ordering of digraphs in $\mathcal{D}_n$ and finds the extremal energy and iota energy, where $\mathcal{D}_n$ is the class of vertex-disjoint bicyclic digraphs of order $n$. Motivated by Yang and Wang \cite{XL2019}, we consider the problem of finding the ordering of sidigraphs in $\mathcal{D}_n^s$ with respect to energy and also investigate extremal energy of sidigraphs in this class, where $\mathcal{D}_n^s$ is the class of vertex-disjoint bicyclic sidigraphs of order $n$. 
\section{Some results and notations}
 Let $p,q\geq 2$ and $D_n^s[p,q]$ be the disjoint union of directed cycles $C_p$ and $C_q$ and $D_n^s[{\boldsymbol{p}},{\boldsymbol{q}}]$ be the disjoint union of directed cycles ${\boldsymbol{C}_p}$ and ${\boldsymbol{C}_q}$. Also suppose $D_n^s[{\boldsymbol{p}},q]$ denotes the disjoint union of directed cycles ${\boldsymbol{C}_p}$ and $C_q$ and $D_n^s[p,{\boldsymbol{q}}]$ denotes the disjoint union of directed cycles $C_p$ and ${\boldsymbol{C}_q}$ and $D_n^s[\mathfrak{p},\mathfrak{q}]$ denotes the disjoint union of directed cycles $\mathcal{C}_p$ and $\mathcal{C}_q$. Let $\mathcal{D}_n^s[p,q] = \left\{D_n^s[p,q],\:D_n^s[{\boldsymbol{p}},{\boldsymbol{q}}],\:D_n^s[{\boldsymbol{p}},q],\: D_n^s[p,{\boldsymbol{q}}]\right\}$.

Let $S$ be a sidigraph with eigenvalues $\rho_1, \dots,\rho_n$. Then energy of $S$ is defined as $E(S)=\sum_{k=1}^n|{\rm Re}(\rho_k)|$,
where ${\rm Re}(\rho_k)$ represents the real value of $\rho_k$.

Relationship between energy of strong components of a sidigraph $S$ and and energy of $S$ is given in the following result.
\begin{lemma}[Pirzada and Bhat \cite{PB2014}]\label{strong}
Let $Q_1,\dots,Q_k$ are strong components of a sidigraph $S$. Then $E(S) = \sum\limits_{j=1}^{k} E(Q_j)$.
\end{lemma}
Pirzada and Bhat \cite{PB2014} gave the following energy formulae for positive and negative directed cycles of order $n\geq2$.
\begin{equation}\label{posi}
E(C_n)=
\left\{
\begin{array}{ll}
2 \cot \frac{\pi}{n} & \mbox{if $n\equiv 0 (\bmod 4)$}\vspace{.2cm}\\
2 \csc \frac{\pi}{n} & \mbox{if $n\equiv 2 (\bmod 4)$}\vspace{.2cm}\\
\csc \frac{\pi}{2n} & \mbox{if $n\equiv 1 (\bmod 2)$,}
\end{array}
\right.
\end{equation}\vspace{.18cm}
\begin{equation}\label{nega}
E({\boldsymbol{C}}_n)=
\left\{
\begin{array}{ll}
2 \csc \frac{\pi}{n} & \mbox{if $n\equiv 0 (\bmod 4)$}\vspace{.2cm}\\
2 \cot \frac{\pi}{n} & \mbox{if $n\equiv 2 (\bmod 4)$}\vspace{.2cm}\\
\csc \frac{\pi}{2n} & \mbox{if $n\equiv 1 (\bmod 2)$.}
\end{array}
\right.
\end{equation}
For any $S\in \mathcal{D}_n^s$, its strong components are: a sidigraph from the set $\mathcal{D}_n^s[p,q]$ and few isolated vertices. Therefore using Lemma \ref{strong}, we can only use the energy of strong components to find the energy ordering in $\mathcal{D}_n^s$.

Using Lemma \ref{strong}, we give the following equations.
\begin{eqnarray*}
E(D_n^s[p,q])&=& E(C_p) +E(C_q),\\
E(D_n^s[{\boldsymbol{p}},{\boldsymbol{q}}])&=& E({\boldsymbol{C}}_p) +E({\boldsymbol{C}}_q),\\
E(D_n^s[{\boldsymbol{p}},q])&=& E({\boldsymbol{C}}_p) +E(C_q),\\
E(D_n^s[p,{\boldsymbol{q}}])&=& E(C_p) +E({\boldsymbol{C}}_q).	
\end{eqnarray*}

Let $n >4$. In Lemmas $2.2$$\sim$$2.8$, we give some results about the monotonicity of some functions which will be used to find the energy ordering of sidigraphs in $\mathcal{D}_n^s$.
\begin{lemma}[Farooq et al. \cite{FMA2017}]\label{lem1}
Suppose $f(z)= 2(\cot\frac{\pi}{z} + \cot\frac{\pi}{n-z})$. For $z \in [2, \frac{n}{2}]$, $f(z)$ is increasing and for $z\in [\frac{n}{2}, n-2]$, $f(z)$ is decreasing.
\end{lemma}
\begin{lemma}[Yang and Wang \cite{XL2019}]\label{lem2}
Let $f(z) = 2(\csc\frac{\pi}{z} +\cot\frac{\pi}{n-z})$. For $z\in [2,n-2]$, $f(z)$ is decreasing.
\end{lemma}
\begin{lemma}[Yang and Wang \cite{XL2019}]\label{lem3}
Suppose $f(z) = 2(\csc\frac{\pi}{z} + \csc\frac{\pi}{n-z})$. For $z\in [2,\frac{n}{2}]$, $f(z)$ is decreasing.
\end{lemma}
\begin{lemma}[Farooq et al. \cite{FMA2017}]\label{lem4}
Let $f(z) = z\sin\frac{\pi}{z}$. For $z\in [2,\infty)$, $f(z)$ is increasing.
\end{lemma}
\begin{lemma}[Yang and Wang \cite{XL2019}]\label{lem5}
Suppose $f(z)= \frac{\pi}{z^2}\cos\frac{\pi}{z}\csc^2\frac{\pi}{z}$. For $z\in [2,n-2]$, $f(z)$ is increasing.
\end{lemma}
The proof of next lemma is similar to the proof of Lemma \ref{lem5} and is thus omitted.
\begin{lemma}\label{lemn1}
Suppose $f(z)= \frac{\pi}{z^2}\cos\frac{\pi}{2z}\csc^2\frac{\pi}{2z}$ and $g(z)=\frac{\pi}{z^2}\cos\frac{\pi}{z}\csc^2\frac{\pi}{z}$. For $z\in [2,\infty)$, $f(z)$ and $g(z)$ are increasing.
\end{lemma}
Now we prove the following results.
\begin{lemma}\label{lem6}
Suppose $f(z) = 2(\cot\frac{\pi}{z}+ \csc\frac{\pi}{n-z})$. For $z\in [2,n-2]$, $f(z)$ is increasing.
\end{lemma}
\begin{proof}
To prove the result, we will show that for all $z\in[2,n-2]$, $f'(z) \geq 0$.

Now\vspace{.2cm}
\begin{equation}\label{e1}
f'(z) = 2\bigg(\frac{\pi}{z^2} \csc^2\frac{\pi}{z}-\frac{\pi}{(n-z)^2} \csc\frac{\pi}{n-z} \cot\frac{\pi}{n-z}\bigg)
\end{equation}
To prove $f'(z)\geq0$, we divide the interval in two parts. Firstly let $z\in[\frac{n}{2}, n-2]$. Then $z\geq n-z$. By Lemma \ref{lem5}, we know that for $z\in  [2,n-2]$, $\frac{\pi}{z^2}\cos\frac{\pi}{z}\csc^2\frac{\pi}{z}$ is increasing. Therefore $
\frac{\pi}{(n-z)^2}\csc\frac{\pi}{n-z}\cot\frac{\pi}{n-z}=\frac{\pi}{(n-z)^2}\csc^2\frac{\pi}{n-z}\cos\frac{\pi}{n-z}
\leq \frac{\pi}{z^2} \csc^2\frac{\pi}{z}\cos\frac{\pi}{z}
< \frac{\pi}{z^2}\csc^2\frac{\pi}{z}.$ 
Hence using \eqref{e1}, $f'(z)\geq0$ for $z\in[\frac{n}{2}, n-2]$.

Now let $z\in[2,\frac{n}{2}]$. Then $z\leq n-z$. By Lemma \ref{lem4}, we know that $z\sin\frac{\pi}{z}$ is strictly increasing on $[2,\infty)$. We have $z\sin\frac{\pi}{z} \leq (n-z) \sin\frac{\pi}{n-z}$. From this, we get $\frac{1}{z}\csc\frac{\pi}{n-z} -\frac{1}{n-z}\csc\frac{\pi}{n-z} \geq 0$. Consider
\begin{equation*}
\frac{\pi}{z^2}\csc^2\frac{\pi}{z}-\frac{\pi}{(n-z)^2}\csc^2\frac{\pi}{n-z}=\pi\bigg(\frac{1}{z}\csc\frac{\pi}{z}+\frac{1}{n-z}\csc\frac{\pi}{n-z}\bigg)\bigg(\frac{1}{z}\csc\frac{\pi}{z}-\frac{1}{n-z}\csc\frac{\pi}{n-z}\bigg).
\end{equation*}
Clearly $\frac{1}{z}\csc\frac{\pi}{z}+\frac{1}{n-z}\csc\frac{\pi}{n-z} >0$ and $\frac{1}{z}\csc\frac{\pi}{z} -\frac{1}{n-z}\csc\frac{\pi}{n-z} \geq 0$. Hence $\frac{\pi}{z^2}\csc^2\frac{\pi}{z}-\frac{\pi}{(n-z)^2}\csc^2\frac{\pi}{n-z} \geq0$. This implies that
$\frac{\pi}{(n-z)^2}\csc\frac{\pi}{n-z}\cot\frac{\pi}{n-z}= \frac{\pi}{(n-z)^2}\cos\frac{\pi}{n-z}\csc^2\frac{\pi}{n-z}
<\frac{\pi}{(n-z)^2}\csc^2\frac{\pi}{n-z}
< \frac{\pi}{z^2}\csc^2\frac{\pi}{z}$.
Hence using \eqref{e1} and all these facts, $f'(z)\geq0$ for $z\in[2,\frac{n}{2}]$. Thus $f'(z)\geq0$ for $z\in[2,n-2]$. This proves the result.
\end{proof}
\begin{lemma}\label{lemn2}
Let $z\in[2,n-2]$. The following holds.
\begin{itemize}
	\item [$(1)$] Let $f(z)= \csc \frac{\pi}{2z}+\csc\frac{\pi}{2(n-z)}$ is decreasing on $[2,\frac{n}{2}]$ and increasing on $[\frac{n}{2},n-2]$.
		\item [$(2)$] The function $f(z)= 2\csc \frac{\pi}{z}+\csc\frac{\pi}{2(n-z)}$ is decreasing on $[2,\frac{2n}{3}]$ and increasing on $[\frac{2n}{3},n-2]$.
		\item [$(3)$] The function $f(z)= \csc \frac{\pi}{2z}+2\csc\frac{\pi}{(n-z)}$ is decreasing on $[2,\frac{n}{3}]$ and increasing on $[\frac{n}{3},n-2]$.
\end{itemize}
\end{lemma}
\begin{proof}
$(1).$ To show that $f(z)$ is decreasing on $[2,\frac{n}{2}]$, it is sufficient to show that $f'(z) \leq 0$.\vspace{.1cm}

Since $z \leq (n-z)$ for $z\in [2,\frac{n}{2}]$, therefore using Lemma \ref{lemn1}, we get
\begin{eqnarray*}
f'(z) &=& \frac{\pi}{2z^2} \csc^2\frac{\pi}{2z} \cos \frac{\pi}{2z}-\frac{\pi}{2(n-z)^2} \csc^2\frac{\pi}{2(n-z)} \cos \frac{\pi}{2(n-z)}\\
&\leq& \frac{\pi}{2(n-z)^2} \csc^2\frac{\pi}{2(n-z)} \cos \frac{\pi}{2(n-z)}-\frac{\pi}{2(n-z)^2} \csc^2\frac{\pi}{2(n-z)} \cos \frac{\pi}{2(n-z)}=0
\end{eqnarray*}
Hence $f(z)$ is decreasing on $[2,\frac{n}{2}]$.\vspace{.1cm}

Now we will show that $f'(z)\geq 0$. Since $z \geq (n-z)$ for $z\in [\frac{n}{2},n-2]$, therefore using Lemma \ref{lemn1}, we obtain
\begin{eqnarray*}
	f'(z) &=& \frac{\pi}{2z^2} \csc^2\frac{\pi}{2z} \cos \frac{\pi}{2z}-\frac{\pi}{2(n-z)^2} \csc^2\frac{\pi}{2(n-z)} \cos \frac{\pi}{2(n-z)}\\
	&\geq& \frac{\pi}{2(n-z)^2} \csc^2\frac{\pi}{2(n-z)} \cos \frac{\pi}{2(n-z)}-\frac{\pi}{2(n-z)^2} \csc^2\frac{\pi}{2(n-z)} \cos \frac{\pi}{2(n-z)}=0
\end{eqnarray*}
Hence $f(z)$ is increasing on $[\frac{n}{2},n-2]$.\\\vspace{.1cm}
$(2.)$ To show that $f(z)$ is decreasing on $[2,\frac{2n}{3}]$, it is enough to prove that $f'(z) \leq 0$.\vspace{.1cm}

Since $z \leq 2(n-z)$ for $z\in [2,\frac{2n}{3}]$, therefore using Lemma \ref{lemn1}, we get
\begin{eqnarray*}
	f'(z) &=& \frac{2\pi}{z^2} \csc^2\frac{\pi}{z} \cos \frac{\pi}{z}-\frac{\pi}{2(n-z)^2} \csc^2\frac{\pi}{2(n-z)} \cos \frac{\pi}{2(n-z)}\\
	&\leq& \frac{2\pi}{4(n-z)^2} \csc^2\frac{\pi}{2(n-z)} \cos \frac{\pi}{2(n-z)}-\frac{\pi}{2(n-z)^2} \csc^2\frac{\pi}{2(n-z)} \cos \frac{\pi}{2(n-z)}=0
\end{eqnarray*}
Hence $f(z)$ is decreasing on $[2,\frac{2n}{3}]$. \vspace{.1cm}

Now we will show that $f'(z)\geq 0$. Since $z \geq 2(n-z)$ for $z\in [\frac{2n}{3},n-2]$, therefore using Lemma \ref{lemn1}, we obtain
\begin{eqnarray*}
	f'(z) &=& \frac{2\pi}{z^2} \csc^2\frac{\pi}{z} \cos \frac{\pi}{z}-\frac{\pi}{2(n-z)^2} \csc^2\frac{\pi}{2(n-z)} \cos \frac{\pi}{2(n-z)}\\
	&\geq& \frac{2\pi}{4(n-z)^2} \csc^2\frac{\pi}{2(n-z)} \cos \frac{\pi}{2(n-z)}-\frac{\pi}{2(n-z)^2} \csc^2\frac{\pi}{2(n-z)} \cos \frac{\pi}{2(n-z)}=0
\end{eqnarray*}
Hence $f(z)$ is increasing on $[\frac{2n}{3},n-2]$.

Analogously $(3)$ can be proved.
\end{proof}
\begin{lemma}\label{lemn3}
Suppose $f(z)= 2\cot\frac{\pi}{z} + \csc \frac{\pi}{2(n-z)}$. For $z \in [2,n-2]$, $f(z)$ is increasing.
\end{lemma}
\begin{proof}
We will show that $f'(z) \geq 0$ for $z \in [2,n-2]$.\vspace{.1cm}

Since $\cos\frac{\pi}{z} \leq 1$ and $z \geq 2(n-z)$ for $z\in[\frac{2n}{3},n-2]$, therefore by Lemma \ref{lemn1}, we have
\begin{eqnarray*}
f'(z) &=& \frac{2\pi}{z^2} \csc^2\frac{\pi}{z}-\frac{\pi}{2(n-z)^2} \csc^2\frac{\pi}{2(n-z)} \cos \frac{\pi}{2(n-z)}\\
&\geq&  \frac{2\pi}{z^2} \csc^2\frac{\pi}{z}\cos\frac{\pi}{z}-\frac{\pi}{2(n-z)^2} \csc^2\frac{\pi}{2(n-z)} \cos \frac{\pi}{2(n-z)}\\ 
&\geq& \frac{2\pi}{4(n-z)^2} \csc^2\frac{\pi}{2(n-z)} \cos \frac{\pi}{2(n-z)}-\frac{\pi}{2(n-z)^2} \csc^2\frac{\pi}{2(n-z)} \cos \frac{\pi}{2(n-z)}=0	
\end{eqnarray*}
Also $-\cos\frac{\pi}{z} \geq -1$ and $z \leq 2(n-z)$ for $z\in[2,\frac{2n}{3}]$, therefore by proof of Lemma $2.4$ \cite{FMA2017}, we see that 
\begin{eqnarray*}
	f'(z) &=& \frac{2\pi}{z^2} \csc^2\frac{\pi}{z}-\frac{\pi}{2(n-z)^2} \csc^2\frac{\pi}{2(n-z)} \cos \frac{\pi}{2(n-z)}\\
	&\geq&  \frac{2\pi}{z^2} \csc^2\frac{\pi}{z}-\frac{\pi}{2(n-z)^2} \csc^2\frac{\pi}{2(n-z)} \geq 0. 	
\end{eqnarray*}
Hence $f(z)$ is increasing on $[2,n-2]$.
\end{proof}
\begin{lemma}[Farooq et al. \cite{FMA2017}]\label{lem7}
	If $z\in (0,\frac{\pi}{2}]$ then
	\begin{equation*}
	\frac{1}{z}-0.429z \leq \cot z \leq \frac{1}{z}-\frac{z}{3}.
	\end{equation*}	
\end{lemma}
For $ 0 < z < \frac{\pi}{2}$, we have
\begin{equation}\label{e2}
z-\frac{z^3}{3!} \leq \sin z \leq z.
\end{equation}
Khan et al. \cite{KFAA} prove the following result.
\begin{lemma}[Khan et al. \cite{KFAA}]\label{inequality}
Suppose $z,d,e$ be real numbers such that $z\geq d > 0$ and $e > 0$. Then
\begin{equation*}
\frac{\pi z}{ez^2-\pi^2} \leq \frac{\pi d}{ed^2-\pi^2}.
\end{equation*}
\end{lemma}
\section{Energy ordering}
 Sidigraphs in $\mathcal{D}_n^s$ are classified into three categories: the sidigraphs whose cycles are of even length, the sidigraphs whose cycles are of odd length and the sidigraphs whose one cycle is of even length and one is of odd length. In the following section, we separately investigate energy ordering in all three categories and find maximal energy.  
\subsection{Both cycles of even length}
In this section, we investigate energy ordering of vertex-disjoint bicyclic sidigraphs whose both cycles are of even length.

Yang and Wang \cite{XL2019} gave the following corollary about those bicyclic sidigraphs in $\mathcal{D}_n^s$ whose both cycles are positive. For details, see \cite{XL2019}.
\begin{corollary}[Yang and Wang \cite{XL2019}]\label{cor1}
Let $n\geq 4$, $r\equiv0(\bmod~2)$, $n-r \geq2$ and $r\in[2,n-2]$. Then it holds that:
\begin{itemize}
	\item[$(i)$] If $n\equiv0(\bmod~2)$, then
	$E(D_n^s[2, n-2]) \geq E(D_n^s[r, n-r]).$
	\item[$(ii)$] If $n\equiv1(\bmod~2)$, then 
	$E(D_n^s[2, n-3]) \geq E(D_n^s[r, n-r-1]).$
\end{itemize}
\end{corollary}
Now we give the energy ordering of those bicyclic sidigraphs in $\mathcal{D}_n^s$ whose one cycle is of positive sign and other cycle is of negative sign.
\begin{lemma}\label{l1}
Let $n>5$ and $n\equiv2(\bmod~4)$. Take $r\in[2,n-2]$ with $r\equiv0(\bmod~2)$ and $n-r\equiv0(\bmod~2)$. Then maximum value of $E(D_n^s[r,{\boldsymbol{n-r}}])$ is attained at $r=2$ and maximum value of $E(D_n^s[{\boldsymbol{r}},n-r])$ is attained at $r=4$. Therefore the following energy ordering holds: 
\begin{itemize}
	\item [$(i)$] If $\frac{n}{2}-3 \equiv 2(\bmod~4)$ then
	\begin{eqnarray*}
	&&E(D_n^s[2,{\boldsymbol{n-2}}]) > E(D_n^s[6,{\boldsymbol{n-6}}])> \dots > E(D_n^s[\frac{n-6}{2},{\boldsymbol{\frac{n+6}{2}}}])\\
	&>& E(D_n^s[\frac{n-2}{2},{\boldsymbol{\frac{n+2}{2}}}])> E(D_n^s[\frac{n-10}{2},{\boldsymbol{\frac{n+10}{2}}}]) >\dots> E(D_n^s[4,{\boldsymbol{n-4}}]). 	
	\end{eqnarray*}
Also
\begin{eqnarray*}
&&E(D_n^s[{\boldsymbol{4}},n-4]) > E(D_n^s[{\boldsymbol{8}},n-8])> \dots > E(D_n^s[{\boldsymbol{\frac{n-2}{2}}},\frac{n+2}{2}])\\
&>& E(D_n^s[{\boldsymbol{\frac{n-6}{2}}},\frac{n+6}{2}])> E(D_n^s[{\boldsymbol{\frac{n-14}{2}}},\frac{n+14}{2}]) >\dots> E(D_n^s[{\boldsymbol{2}},n-2]).	
\end{eqnarray*}
\item[$(ii)$] If $\frac{n}{2}-3 \equiv 0(\bmod~4)$ then
\begin{eqnarray*}
	&&E(D_n^s[2,{\boldsymbol{n-2}}]) > E(D_n^s[6,{\boldsymbol{n-6}}])> \dots > E(D_n^s[\frac{n-2}{2},{\boldsymbol{\frac{n+2}{2}}}])\\
	&>& E(D_n^s[\frac{n-6}{2},{\boldsymbol{\frac{n+6}{2}}}])> E(D_n^s[\frac{n-14}{2},{\boldsymbol{\frac{n+14}{2}}}]) >\dots> E(D_n^s[4,{\boldsymbol{n-4}}]).	
\end{eqnarray*}
Also
\begin{eqnarray*}
	&&E(D_n^s[{\boldsymbol{4}},n-4]) > E(D_n^s[{\boldsymbol{8}},n-8])> \dots > E(D_n^s[{\boldsymbol{\frac{n-6}{2}}},\frac{n+6}{2}])\\
	&>& E(D_n^s[{\boldsymbol{\frac{n-2}{2}}},\frac{n+2}{2}])> E(D_n^s[{\boldsymbol{\frac{n-10}{2}}},\frac{n+10}{2}]) >\dots> E(D_n^s[{\boldsymbol{2}},n-2]).
\end{eqnarray*}
\end{itemize}
\end{lemma}
\begin{proof}
$(i)$. If $r\in[2,\frac{n}{2}]$ then $n-r \in[\frac{n}{2},n-2]$ and if $r\in[\frac{n}{2},n-2]$ then $n-r \in[2,\frac{n}{2}]$. Therefore we take $r\in[2,\frac{n}{2}]$.

Let $r\equiv2(\bmod~4)$. Then $n-r\equiv0(\bmod~4)$. Using \eqref{posi} and \eqref{nega}, we have
\begin{eqnarray*}
E(D_n^s[r,{\boldsymbol{n-r}}])&=& 2(\csc\frac{\pi}{r} +\csc\frac{\pi}{n-r}),\\
E(D_n^s[{\boldsymbol{r}},n-r])&=& 2(\cot\frac{\pi}{r} +\cot\frac{\pi}{n-r}).
\end{eqnarray*}
By Lemma \ref{lem1} and Lemma \ref{lem3}, we know that $ 2(\csc\frac{\pi}{r} +\csc\frac{\pi}{n-r})$ is decreasing on $[2, \frac{n}{2}]$ and $2(\cot\frac{\pi}{r} +\cot\frac{\pi}{n-r})$ is increasing on $[2, \frac{n}{2}]$. As $\frac{n}{2}-3 \equiv 2(\bmod~4)$, thus for each $r\equiv2(\bmod~2)$, we obtain
\begin{eqnarray*}
E(D_n^s[2,{\boldsymbol{n-2}}]) &>& E(D_n^s[6,{\boldsymbol{n-6}}])> \dots > E(D_n^s[\frac{n}{2}-3,{\boldsymbol{\frac{n}{2}+3}}])\\
E(D_n^s[{\boldsymbol{\frac{n}{2}-3}},\frac{n}{2}+3])&>&
 E(D_n^s[{\boldsymbol{\frac{n}{2}-7}},\frac{n}{2}+7])> \dots> E(D_n^s[{\boldsymbol{2}},n-2]).
\end{eqnarray*}
Let $r\equiv0(\bmod~4)$. Then $n-r\equiv2(\bmod~4)$. Equations \eqref{posi} and \eqref{nega} give
\begin{eqnarray*}
E(\mathcal{D}_n^s[r,{\boldsymbol{n-r}}])&=& 2(\cot\frac{\pi}{r} +\cot\frac{\pi}{n-r})\\
E(\mathcal{D}_n^s[{\boldsymbol{r}},n-r])&=& 2(\csc\frac{\pi}{r} +\csc\frac{\pi}{n-r}).
\end{eqnarray*}
By Lemma \ref{lem1} and Lemma \ref{lem3}, it is clear that  $2(\cot\frac{\pi}{r} +\cot\frac{\pi}{n-r})$ is increasing on $ [2,\frac{n}{2}]$, and $2(\csc\frac{\pi}{r} +\csc\frac{\pi}{n-r})$ is decreasing on $[2,\frac{n}{2}]$. Hence 
\begin{eqnarray*}
E(D_n^s[\frac{n}{2}-1,{\boldsymbol{\frac{n}{2}+1}}])&>& E(D_n^s[\frac{n}{2}-5,{\boldsymbol{\frac{n}{2}+5}}]) >\dots> E(D_n^s[4,{\boldsymbol{n-4}}]),\\
E(D_n^s[{\boldsymbol{4}},n-4]) &>& E(D_n^s[{\boldsymbol{8}},n-8])> \dots > E(D_n^s[{\boldsymbol{\frac{n}{2}-1}},\frac{n}{2}+1]).
\end{eqnarray*}
Now we will show that $E(D_n^s[\frac{n}{2}-3,{\boldsymbol{\frac{n}{2}+3}}]) > E(D_n^s[\frac{n}{2}-1,{\boldsymbol{\frac{n}{2}+1}}])$. Lemma \ref{lem7} and \eqref{e2} yields $E(D_n^s[\frac{n}{2}-1,{\boldsymbol{\frac{n}{2}+1}}])= 2(\cot\frac{\pi}{\frac{n}{2}-1} +\cot\frac{\pi}{\frac{n}{2}+1})
\leq \frac{2n}{\pi}-\frac{4\pi}{3}(\frac{1}{n-2}+\frac{1}{n+2})
= \frac{2n}{\pi} -\frac{8n\pi}{3(n^2-4)}
< \frac{2n}{\pi} = 2(\frac{n-6}{2\pi}+\frac{n+6}{2\pi})
\leq 2(\csc\frac{\pi}{\frac{n}{2}-3}+\csc\frac{\pi}{\frac{n}{2}+3})
= E(D_n^s[\frac{n}{2}-3,{\boldsymbol{\frac{n}{2}+3}}]).$
Similarly $E(D_n^s[{\boldsymbol{\frac{n}{2}-1}},\frac{n}{2}+1]) > E(D_n^s[{\boldsymbol{\frac{n}{2}-3}},\frac{n}{2}+3])$. Thus we get the required energy ordering.

One can prove $(ii)$ analogously.
\end{proof}
\begin{lemma}\label{l2}
Suppose $n>5$, $n\equiv0(\bmod~4)$ and $r\in[2,n-2]$. If $r\equiv0(\bmod~4)$ then $E(D_n^s[r,{\boldsymbol{n-r}}])$ attains maximum value at $r=4$ and if $r\equiv2(\bmod~4)$ then $E(D_n^s[r,{\boldsymbol{n-r}}])$ attains maximum value at $r=2$ . Therefore the following energy ordering holds:
\begin{itemize}
\item[$(i)$] If $r\equiv0(\bmod~4)$ then
\begin{equation*}
E(D_n^s[n-4,{\boldsymbol{4}}])> E(D_n^s[n-8,{\boldsymbol{8}}])>\dots> E(D_n^s[8,{\boldsymbol{n-8}}])> E(D_n^s[4,{\boldsymbol{n-4}}]).
\end{equation*}	
\item[$(ii)$] If $r\equiv2(\bmod~4)$ then
\begin{eqnarray*}
&&E(D_n^s[2,{\boldsymbol{n-2}}])> E(D_n^s[6,{\boldsymbol{n-6}}])> E(D_n^s[10,{\boldsymbol{n-10}}])> E(D_n^s[14,{\boldsymbol{n-14}}])\\
&&>\dots> E(D_n^s[n-6,{\boldsymbol{6}}])> E(D_n^s[n-2,{\boldsymbol{2}}]).
\end{eqnarray*}	
\end{itemize}
\end{lemma}
\begin{proof}
$(i)$. If $n\equiv0(\bmod~4)$ and $r\equiv0(\bmod~4)$ then $n-r\equiv0(\bmod~4)$. Using \eqref{posi} and \eqref{nega}, we obtain
\begin{equation*}
E(D_n^s[r,{\boldsymbol{n-r}}])= 2(\cot\frac{\pi}{r}+\csc\frac{\pi}{n-r}).
\end{equation*}
Lemma \ref{lem6} gives that $2(\cot\frac{\pi}{r}+\csc\frac{\pi}{n-r})$ is increasing on $[2,n-2]$. As $r\equiv0(\bmod~4)$, therefore we have
\begin{equation*}
E(D_n^s[n-4,{\boldsymbol{4}}])> E(D_n^s[n-8,{\boldsymbol{8}}])>\dots> E(D_n^s[8,{\boldsymbol{n-8}}])> E(D_n^s[4,{\boldsymbol{n-4}}]).
\end{equation*}
$(ii)$. If $n\equiv 0(\bmod~4)$ and $r\equiv2(\bmod~4)$ then $n-r\equiv2(\bmod~4)$. Using \eqref{posi} and \eqref{nega}, we have
\begin{equation*}
E(D_n^s[r,{\boldsymbol{n-r}}])= 2(\csc\frac{\pi}{r}+\cot\frac{\pi}{n-r}).
\end{equation*}
Lemma \ref{lem2} yields that $2(\csc\frac{\pi}{r}+\cot\frac{\pi}{n-r})$ is decreasing on $[2,n-2]$. As $r\equiv2(\bmod~4)$, therefore we have
\begin{eqnarray*}
	&&E(D_n^s[2,{\boldsymbol{n-2}}])> E(D_n^s[6,{\boldsymbol{n-6}}])> E(D_n^s[10,{\boldsymbol{n-10}}])> E(D_n^s[14,{\boldsymbol{n-14}}])\\
	&&>\dots> E(D_n^s[n-6,{\boldsymbol{6}}])> E(D_n^s[n-2,{\boldsymbol{2}}]).
\end{eqnarray*}		
The proof is complete.
\end{proof}
If we replace $n$ by $n-1$ in Lemma \ref{l2}, we get the following result.
\begin{lemma}\label{l3}
	Let $n>5$ and $n\equiv1(\bmod~4)$. Take $r\in[2,n-2]$ with $r\equiv0(\bmod~2)$ and $n-1-r\equiv0(\bmod~2)$. If $r\equiv0(\bmod~4)$ then $E(D_n^s[r, {\boldsymbol{n-r-1}}])$ attains maximum value at $r=4$ and if $r\equiv0(\bmod~4)$ then $E(D_n^s[r, {\boldsymbol{n-r-1}}])$ attains maximum value at $r=2$. Therefore the following energy ordering holds:
	\begin{itemize}
		\item[$(i)$] If $r\equiv0(\bmod~4)$ then
		\begin{equation*}
		E(D_n^s[n-5,{\boldsymbol{4}}])> E(D_n^s[n-9,{\boldsymbol{8}}])>\dots> E(D_n^s[8,{\boldsymbol{n-9}}])> E(D_n^s[4,{\boldsymbol{n-5}}]).
		\end{equation*}	
		\item[$(ii)$] If $r\equiv2(\bmod~4)$ then
		\begin{equation*}
		E(D_n^s[2,{\boldsymbol{n-3}}])> E(D_n^s[6,{\boldsymbol{n-7}}])>\dots> E(D_n^s[n-7,{\boldsymbol{6}}])> E(D_n^s[n-3,{\boldsymbol{2}}]).
		\end{equation*}	
	\end{itemize}
\end{lemma}
\begin{lemma}\label{l4}
Let  $n>5$, $n\equiv3(\bmod~4)$ and $r\in[2,n-2]$ satisfying $r\equiv0(\bmod~2)$, $n-r-1\equiv0(\bmod~2)$. Then maximum value of $E(D_n^s[r,{\boldsymbol{n-r-1}}])$ is attained at $r=2$ and maximum value of $E(D_n^s[{\boldsymbol{r}},n-r-1])$ is attained at $r=4$. Therefore we obtain the energy ordering: 
\begin{itemize}
	\item [$(i)$] If $\frac{n-1}{2}-3 \equiv 2(\bmod~4)$, then
	\begin{eqnarray*}
		&&E(D_n^s[2,{\boldsymbol{n-3}}]) > E(D_n^s[6,{\boldsymbol{n-7}}])> E(D_n^s[10,{\boldsymbol{n-11}}])> E(D_n^s[14,{\boldsymbol{n-15}}])\\
		&&> \dots > E(D_n^s[\frac{n-7}{2},{\boldsymbol{\frac{n+5}{2}}}])\\
		&&> E(D_n^s[\frac{n-3}{2},{\boldsymbol{\frac{n+1}{2}}}])> E(D_n^s[\frac{n-11}{2},{\boldsymbol{\frac{n+9}{2}}}]) >\dots> E(D_n^s[4,{\boldsymbol{n-5}}]).	
	\end{eqnarray*}
	Also
	\begin{eqnarray*}
		&&E(D_n^s[{\boldsymbol{4}},n-5]) > E(D_n^s[{\boldsymbol{8}},n-9])> \dots > E(D_n^s[{\boldsymbol{\frac{n-3}{2}}},\frac{n+1}{2}])\\
		&>& E(D_n^s[{\boldsymbol{\frac{n-7}{2}}},\frac{n+5}{2}])> E(D_n^s[{\boldsymbol{\frac{n-15}{2}}},\frac{n+13}{2}]) >\dots> E(D_n^s[{\boldsymbol{2}},n-3]).	
	\end{eqnarray*}
	\item[$(ii)$] If $\frac{n-1}{2}-3 \equiv 0(\bmod~4)$, then
	\begin{eqnarray*}
		&&E(D_n^s[2,{\boldsymbol{n-3}}]) > E(D_n^s[6,{\boldsymbol{n-7}}])> \dots > E(D_n^s[\frac{n-3}{2},{\boldsymbol{\frac{n+1}{2}}}])\\
		&>& E(D_n^s[\frac{n-7}{2},{\boldsymbol{\frac{n+5}{2}}}])> E(D_n^s[\frac{n-15}{2},{\boldsymbol{\frac{n+13}{2}}}]) >\dots> E(D_n^s[4,{\boldsymbol{n-5}}]).	
	\end{eqnarray*}
	Also
	\begin{eqnarray*}
		&&E(D_n^s[{\boldsymbol{4}},n-5]) > E(D_n^s[{\boldsymbol{8}},n-9])> \dots > E(D_n^s[{\boldsymbol{\frac{n-7}{2}}},\frac{n+5}{2}])\\
		&>& E(D_n^s[{\boldsymbol{\frac{n-3}{2}}},\frac{n+1}{2}])> E(D_n^s[{\boldsymbol{\frac{n-11}{2}}},\frac{n+9}{2}]) >\dots> E(D_n^s[{\boldsymbol{2}},n-3]).
	\end{eqnarray*}
\end{itemize}
\end{lemma}
\begin{proof}
If $r\in[2,\frac{n}{2}]$ then $n-r \in[\frac{n}{2},n-2]$ and if $r\in[\frac{n}{2},n-2]$ then $n-r \in[2,\frac{n}{2}]$. Thus it is enough to consider $r\in[2,\frac{n}{2}]$.

Let $r\equiv2(\bmod~4)$. Then $n-r-1\equiv0(\bmod~4)$. Using \eqref{posi} and \eqref{nega}, we have
\begin{eqnarray*}
	E(D_n^s[r,{\boldsymbol{n-r-1}}])&=& 2(\csc\frac{\pi}{r} +\csc\frac{\pi}{n-r-1}),\\
	E(D_n^s[{\boldsymbol{r}},n-r-1])&=& 2(\cot\frac{\pi}{r} +\cot\frac{\pi}{n-r-1}).
\end{eqnarray*}
Let $r\equiv0(\bmod~4)$. Then $n-r-1\equiv2(\bmod~4)$. Equations \eqref{posi} and \eqref{nega} give
\begin{eqnarray*}
	E(D_n^s[r,{\boldsymbol{n-r-1}}])&=& 2(\cot\frac{\pi}{r} +\cot\frac{\pi}{n-r-1})\\
	E(D_n^s[{\boldsymbol{r}},n-r-1])&=& 2(\csc\frac{\pi}{r} +\csc\frac{\pi}{n-r-1}).
\end{eqnarray*}
Hence by changing $n$ to $n-1$ in Lemma \ref{l1}, we get the required ordering.	
\end{proof}
From Lemmas \ref{l1}$\sim$\ref{l4}, the following result is obtained.
\begin{corollary}\label{cor2}
	Suppose $n> 5$, $r\equiv0(\bmod~2)$ and $n-r \geq2$. Then we have the following:
	\begin{itemize}
		\item[$(i)$] Let $n\equiv0(\bmod~4)$.
		\begin{itemize}
			\item[$(a)$] If $r\equiv0(\bmod~4)$ then
			$E(D_n^s[{\boldsymbol{4}}, n-4]) \geq E(D_n^s[{\boldsymbol{r}}, n-r]).$
			\item[$(b)$] If $r\equiv2(\bmod~4)$ then
			$E(D_n^s[{\boldsymbol{n-2}}, 2]) \geq E(D_n^s[{\boldsymbol{r}}, n-r]).$
		\end{itemize}
		\item[$(ii)$] If $n\equiv2(\bmod~4)$ then
		$E(D_n^s[2, {\boldsymbol{n-2}}]) \geq E(\mathcal{D}_n^s[r, {\boldsymbol{n-r}}])$ and $E(D_n^s[{\boldsymbol{4}}, n-4]) \geq E(D_n^s[{\boldsymbol{r}}, n-r]).$
		\item[$(iii)$] If $n\equiv3(\bmod~4)$ then $
			E(D_n^s[2, {\boldsymbol{n-3}}]) \geq E(D_n^s[r, {\boldsymbol{n-r-1}}])$ and $
			E(D_n^s[{\boldsymbol{4}}, n-5]) \geq E(D_n^s[{\boldsymbol{r}}, n-r-1]).$
		\item[$(iv)$] Let $n\equiv1(\bmod~4)$.
		\begin{itemize}
			\item[$(a)$] If $r\equiv0(\bmod~4)$ then $
			E(D_n^s[{\boldsymbol{4}}, n-5]) \geq E(D_n^s[{\boldsymbol{r}}, n-r-1]). $
			\item[$(b)$] If $r\equiv2(\bmod~4)$ then
			$E(D_n^s[{\boldsymbol{n-3}}, 2]) \geq E(D_n^s[{\boldsymbol{r}}, n-r-1]). $
		\end{itemize}  
	\end{itemize}
\end{corollary}
Now we give the energy ordering of those bicyclic sidigraphs in $\mathcal{D}_n^s$ whose signs of both cycles are negative.
\begin{lemma}\label{lem12}
Let $n>5$ and $n\equiv0(\bmod~4)$. Take $r\in[2,n-2]$ satisfying $r\equiv0(\bmod~2)$ and $n-r\equiv0(\bmod~2)$. Then $E(D_n^s[{\boldsymbol{r}},{\boldsymbol{n-r}}])$ has maximum value at $r=4$. Therefore the following energy ordering holds:
\begin{itemize}
	\item [$(i)$] If $\frac{n}{2} \equiv 0(\bmod~4)$ then
	\begin{eqnarray*}
		&&E(D_n^s[{\boldsymbol{4}}, {\boldsymbol{n-4}}]) > E(D_n^s[{\boldsymbol{8}}, {\boldsymbol{n-8}}])> \dots> E(D_n^s[{\boldsymbol{\frac{n}{2}}}, {\boldsymbol{\frac{n}{2}}}])\\
		&>& E(D_n^s[{\boldsymbol{\frac{n-4}{2}}}, {\boldsymbol{\frac{n+4}{2}}}])> E(D_n^s[{\boldsymbol{\frac{n-12}{2}}}, {\boldsymbol{\frac{n+12}{2}}}])> \dots> 	E(D_n^s[{\boldsymbol{2}}, {\boldsymbol{n-2}}]).   
	\end{eqnarray*}
	\item [$(i)$] If $\frac{n}{2} \equiv 2(\bmod~4)$ then
	\begin{eqnarray*}
		&&E(D_n^s[{\boldsymbol{4}}, {\boldsymbol{n-4}}]) > E(D_n^s[{\boldsymbol{8}}, {\boldsymbol{n-8}}])> \dots> E(D_n^s[{\boldsymbol{\frac{n-4}{2}}}, {\boldsymbol{\frac{n+4}{2}}}])\\
		&>& E(D_n^s[{\boldsymbol{\frac{n}{2}}}, {\boldsymbol{\frac{n}{2}}}])> E(D_n^s[{\boldsymbol{\frac{n-8}{2}}}, {\boldsymbol{\frac{n+8}{2}}}])> \dots> 	E(D_n^s[{\boldsymbol{2}}, {\boldsymbol{n-2}}]).
	\end{eqnarray*}
\end{itemize}
\end{lemma}
\begin{proof}
$(i)$. If $r\in[2,\frac{n}{2}]$ then $n-r \in[\frac{n}{2},n-2]$ and if $r\in[\frac{n}{2},n-2]$ then $n-r \in[2,\frac{n}{2}]$. Thus it is enough to consider $r\in[2,\frac{n}{2}]$. 

Let $r\equiv0(\bmod~4)$. Then $n-r \equiv0(\bmod~4)$. Using \eqref{nega}, we get
\begin{equation*}
E(D_n^s[{\boldsymbol{r}},{\boldsymbol{n-r}}])= 2(\csc\frac{\pi}{r}+\csc\frac{\pi}{n-r}).
\end{equation*}
Lemma \ref{lem3} tells us that $2(\csc\frac{\pi}{r}+\csc\frac{\pi}{n-r})$ is decreasing on $[2,\frac{n}{2}]$. As $\frac{n}{2} \equiv 0(\bmod~4)$, for each $r\equiv0(\bmod~4)$, it holds that
\begin{equation*}
E(D_n^s[{\boldsymbol{4}}, {\boldsymbol{n-4}}]) > E(D_n^s[{\boldsymbol{8}}, {\boldsymbol{n-8}}])> \dots> E(D_n^s[{\boldsymbol{\frac{n}{2}}}, {\boldsymbol{\frac{n}{2}}}]).
\end{equation*}
Now let $r\equiv 2(\bmod~4)$. Then $n-r \equiv 2(\bmod~4)$. Using \eqref{nega}, we obtain
\begin{equation*}
E(D_n^s[{\boldsymbol{r}},{\boldsymbol{n-r}}])= 2(\cot\frac{\pi}{r}+\cot\frac{\pi}{n-r}).
\end{equation*}
Lemma \ref{lem1} gives that $2(\cot\frac{\pi}{r}+\cot\frac{\pi}{n-r})$ is increasing on $[2,\frac{n}{2}]$. As $\frac{n}{2} \equiv 0(\bmod~4)$, we have
\begin{equation*}
E(D_n^s[{\boldsymbol{\frac{n}{2}-2}}, {\boldsymbol{\frac{n}{2}+2}}])> E(D_n^s[{\boldsymbol{\frac{n}{2}-6}}, {\boldsymbol{\frac{n}{2}+6}}])> \dots> 	E(D_n^s[{\boldsymbol{2}}, {\boldsymbol{n-2}}]).   
\end{equation*}
Next we prove that $E(D_n^s[{\boldsymbol{\frac{n}{2}}}, {\boldsymbol{\frac{n}{2}}}]) > E(D_n^s[{\boldsymbol{\frac{n}{2}-2}}, {\boldsymbol{\frac{n}{2}+2}}])$. Lemma \ref{lem7} and \eqref{e2} give \\
$
E(D_n^s[{\boldsymbol{\frac{n}{2}-2}}, {\boldsymbol{\frac{n}{2}+2}}])=2(\cot\frac{\pi}{\frac{n}{2}-2}+ \cot\frac{\pi}{\frac{n}{2}+2})
\leq\frac{2n}{\pi} -\frac{4\pi}{3}(\frac{1}{n-4}+ \frac{1}{n+4})
= \frac{2n}{\pi} -\frac{8n\pi}{3(n^2-16)}
< \frac{2n}{\pi} \leq 4\csc\frac{\pi}{\frac{n}{2}}
=E(D_n^s[{\boldsymbol{\frac{n}{2}}}, {\boldsymbol{\frac{n}{2}}}]).$
This proves the result. 

Analogously, one can prove $(ii)$.
\end{proof}
The proof of the next two lemmas are similar to the proofs of Lemma $3.2$ \cite{XL2019} and Lemma $3.4$ \cite{XL2019} and is thus omitted.
\begin{lemma}\label{lem13}
Suppose $n>5$, $n\equiv2(\bmod~4)$. Take $r\in[2,n-2]$ satisfying $r\equiv0(\bmod~2)$ and $n-r\equiv0(\bmod~2)$. Then $E(D_n^s[{\boldsymbol{r}},{\boldsymbol{n-r}}])$ has maximum value at $r=2$. Thus we get the following energy ordering:
\begin{equation*}
E(D_n^s[{\boldsymbol{2}},{\boldsymbol{n-2}}]) > E(D_n^s[{\boldsymbol{6}},{\boldsymbol{n-6}}])>\dots> E(D_n^s[{\boldsymbol{n-4}},{\boldsymbol{4}}]).
\end{equation*}
\end{lemma}
\begin{lemma}\label{lem14}
Let  $n>5$, $n\equiv3(\bmod~4)$. Take $r\in[2,n-2]$ satisfying $r\equiv0(\bmod~2)$ and $n-1-r\equiv0(\bmod~2)$. Then $E(D_n^s[{\boldsymbol{r}},{\boldsymbol{n-r-1}}])$ has maximum value at $r=2$. Therefore the following energy ordering holds:
\begin{equation*}
E(D_n^s[{\boldsymbol{2}},{\boldsymbol{n-3}}]) > E(D_n^s[{\boldsymbol{6}},{\boldsymbol{n-7}}])>\dots> E(D_n^s[{\boldsymbol{n-5}},{\boldsymbol{4}}]).
\end{equation*}
\end{lemma}
By changing $n$ in Lemma \ref{lem12} to $n-1$, we get the following result.
\begin{lemma}\label{lem15}
Suppose $n>5$, $n\equiv1(\bmod~4)$. Take $r\in[2,n-2]$ satisfying $r\equiv0(\bmod~2)$ and $n-1-r\equiv0(\bmod~2)$. Then $E(D_n^s[{\boldsymbol{r}},{\boldsymbol{n-r-1}}])$ has maximum value at $r=4$. Thus we obtain
\begin{itemize}
	\item [$(i)$] If $\frac{n-1}{2} \equiv 0(\bmod~4)$ then
	\begin{eqnarray*}
		&&E(D_n^s[{\boldsymbol{4}}, {\boldsymbol{n-5}}]) > E(D_n^s[{\boldsymbol{8}}, {\boldsymbol{n-9}}])> \dots> E(D_n^s[{\boldsymbol{\frac{n-1}{2}}}, {\boldsymbol{\frac{n-1}{2}}}])\\
		&>& E(D_n^s[{\boldsymbol{\frac{n-5}{2}}}, {\boldsymbol{\frac{n+3}{2}}}])> E(D_n^s[{\boldsymbol{\frac{n-13}{2}}}, {\boldsymbol{\frac{n+11}{2}}}])\\
		&>& \dots> 	E(D_n^s[{\boldsymbol{2}}, {\boldsymbol{n-3}}]).   
	\end{eqnarray*}
	\item [$(i)$] If $\frac{n-1}{2} \equiv 2(\bmod~4)$ then
	\begin{eqnarray*}
		&&E(D_n^s[{\boldsymbol{4}}, {\boldsymbol{n-5}}]) > E(D_n^s[{\boldsymbol{8}}, {\boldsymbol{n-9}}])> \dots> E(D_n^s[{\boldsymbol{\frac{n-5}{2}}}, {\boldsymbol{\frac{n+3}{2}}}])\\
		&>& E(D_n^s[{\boldsymbol{\frac{n-1}{2}}}, {\boldsymbol{\frac{n-1}{2}}}])> E(D_n^s[{\boldsymbol{\frac{n-9}{2}}}, {\boldsymbol{\frac{n+7}{2}}}])> \dots> 	E(D_n^s[{\boldsymbol{2}}, {\boldsymbol{n-3}}]). 
	\end{eqnarray*}
\end{itemize}
\end{lemma}
Now from Lemmas \ref{lem12}$\sim$\ref{lem15}, Corollary \ref{cor3} is obtained.
\begin{corollary}\label{cor3}
Suppose $r\equiv0(\bmod~2)$ and $n> 5$.
\begin{itemize}
	\item[$(i)$] If $n\equiv0(\bmod~4)$ then$
	E(D_n^s[{\boldsymbol{4}}, {\boldsymbol{n-4}}]) \geq E(D_n^s[{\boldsymbol{r}}, {\boldsymbol{n-r}}]).$ 
\item[$(ii)$] If $n\equiv2(\bmod~4)$ then $
E(D_n^s[{\boldsymbol{2}}, {\boldsymbol{n-2}}]) \geq E(D_n^s[{\boldsymbol{r}}, {\boldsymbol{n-r}}]). $
\item[$(iii)$] If $n\equiv3(\bmod~4)$ then $
E(D_n^s[{\boldsymbol{2}}, {\boldsymbol{n-3}}]) \geq E(D_n^s[{\boldsymbol{r}}, {\boldsymbol{n-r-1}}]).$
\item[$(iv)$] If $n\equiv1(\bmod~4)$ then $
E(D_n^s[{\boldsymbol{4}}, {\boldsymbol{n-5}}]) \geq E(D_n^s[{\boldsymbol{r}}, {\boldsymbol{n-r-1}}]). $
\end{itemize}
\end{corollary}
Using Corollaries \ref{cor1}, \ref{cor2} and \ref{cor3}, we give the extremal energy of bicyclic sidigraphs in the class $\mathcal{D}_n^s$.
\begin{theorem}
Let $S\in\mathcal{D}_n^s$ be a sidigraph with even directed cycles.
\begin{itemize}
\item[$(i)$] For $n\equiv0(\bmod~4)$, the maximal energy of $S$ is attained if $S\cong D_n^s[2,n-2]$.
\item[$(ii)$] For $n\equiv2(\bmod~4)$, the maximal energy of $S$ is attained if $S\cong D_n^s[2,{\boldsymbol{n-2]}}$. 
\item[$(iii)$] For $n\equiv1(\bmod~4)$, the maximal energy of $S$ is attained if $S\cong D_n^s[2,n-3]$.
\item[$(iv)$] For $n\equiv3(\bmod~4)$, the maximal energy of $S$ is attained if $S\cong D_n^s[2,{\boldsymbol{n-3}}]$. 
\item[$(v)$] The minimal energy of $S$ is attained if $S\cong D_n^s[{\boldsymbol{2}},{\boldsymbol{2}}]$.
\end{itemize}
\end{theorem}
\begin{proof}
$(i)$. Suppose $n\equiv0(\bmod~4)$. Using \eqref{posi} and \eqref{e2}, we obtain
\begin{align}\label{e3}
\begin{split}
E(D_n^s[2,n-2])&= 2(\csc\frac{\pi}{2} + \csc\frac{\pi}{n-2})\\
&\geq 2 + \frac{2n}{\pi}-\frac{4}{\pi}\\
&\geq \frac{2n}{\pi}+0.7267.
\end{split}
\end{align}
Now Lemmas \ref{lem7}, \ref{inequality} and equations \eqref{posi}, \eqref{nega} give
\begin{align}\label{e4}
\begin{split}
E(D_n^s[{\boldsymbol{4}},n-4])&= 2(\csc\frac{\pi}{4} + \cot\frac{\pi}{n-4})\\
&\leq 2\sqrt{2} + \frac{2n}{\pi}-\frac{8}{\pi}-\frac{2\pi}{3(n-4)}\\
&\leq \frac{2n}{\pi}+0.2819 -\frac{2\pi}{3(n-4)}.
\end{split}
\end{align}
Since $n\equiv0(\bmod~4)$ and $n>5$, therefore $n-4 \geq4$. Using Lemma \ref{inequality} and equations \eqref{nega}, \eqref{e2}, we have
\begin{align}\label{e5}
\begin{split}
E(D_n^s[{\boldsymbol{4}},{\boldsymbol{n-4}}])&= 2(\csc\frac{\pi}{4} + \csc\frac{\pi}{n-4})\\
&\leq 2\sqrt{2} + 2 \bigg(\frac{n-4}{\pi}\bigg)\bigg(\frac{6(n-4)^2}{6(n-4)^2-\pi^2}\bigg)\\
&= 2\sqrt{2} +  \bigg(\frac{2n-8}{\pi}\bigg)\bigg(1+\frac{\pi^2}{6(n-4)^2-\pi^2}\bigg)\\
&=2\sqrt{2} + \frac{2n}{\pi}-\frac{8}{\pi}+\frac{2(n-4)\pi}{6(n-4)^2-\pi^2}\\
&\leq \frac{2n}{\pi}+0.2819+\frac{8\pi}{6(4)^2-\pi^2}\\
&\leq \frac{2n}{\pi}+0.5737.
\end{split}
\end{align}
Also $E(D_n^s[2,{\boldsymbol{n-2}}])= 2+2\cot\frac{\pi}{n-2} < 2+ 2\csc\frac{\pi}{n-2}=  E(D_n^s[2,n-2])$.\vspace{.1cm}

Now take any $H\in \mathcal{D}_n^s$ with even directed cycles $\mathcal{C}_{p}$ and $\mathcal{C}_{q}$. Since $n-p>q$, therefore $E(D_n^s[2,n-2])= E(C_2)+E(C_{n-2})> E(\mathcal{C}_p) +E(\mathcal{C}_{n-p}) > E(\mathcal{C}_p) +E(\mathcal{C}_{q})=E(H)$. Thus $D_n^s[2,n-2]$ has maximal energy.\vspace{.1cm}
  
The proof of part $(ii)$ and part $(iii)$ is similar to the proof of part $(i)$.\vspace{.1cm}

$(iv)$. Using \eqref{posi} and \eqref{nega}, we get $E(D_n^s[2,n-3])=2(\csc\frac{\pi}{2}+\cot\frac{\pi}{n-3}) < 2(\csc\frac{\pi}{2}+\csc\frac{\pi}{n-3}) = E(D_n^s[2,{\boldsymbol{n-3}}])$. Also $E(D_n^s[{\boldsymbol{2}},{\boldsymbol{n-3}}])=2(\cot\frac{\pi}{2}+\csc\frac{\pi}{n-3})= 2(\csc\frac{\pi}{n-3}) < 2(\csc\frac{\pi}{2}+\csc\frac{\pi}{n-3}) = E(D_n^s[2,{\boldsymbol{n-3}}])$. Now for $n=7$, $E(D^s_n[{\boldsymbol{4}},n-5])=2 +2\sqrt{2} = E(D_n^s[2,{\boldsymbol{n-3}}])$. So take $n>7$ and using \eqref{e2}, we have $ E(D_n^s[{\boldsymbol{4}},{\boldsymbol{n-5}}])=2(\csc\frac{\pi}{4}+\csc\frac{\pi}{n-5}) \leq \frac{2n}{\pi} -0.17178 < \frac{2n}{\pi} +0.0901 \leq \frac{2n}{\pi} +2-\frac{6}{\pi} \leq 2(\csc\frac{\pi}{2}+\csc\frac{\pi}{n-3}) = E(D_n^s[2,{\boldsymbol{n-3}}])$.\vspace{.1cm} 

Now take any $H\in \mathcal{D}_n^s$ with even directed cycles $\mathcal{C}_{p}$ and $\mathcal{C}_{q}$. Since $n-p>q$, therefore $E(D_n^s[2,\boldsymbol{n-3}])= E(C_2)+E(\boldsymbol{C}_{n-3})> E(\mathcal{C}_p) +E(\mathcal{C}_{n-p}) > E(\mathcal{C}_p) +E(\mathcal{C}_{q})=E(H)$.
Therefore $D_n^s[2,{\boldsymbol{n-3}}]$ has maximal energy.\vspace{.1cm}

$(v)$. We know $E(D_n^s[{\boldsymbol{2}},{\boldsymbol{2}}])= 4\cot\frac{\pi}{2} =0$. Hence $S$ has minimal energy if $S \cong D_n^s[{\boldsymbol{2}},{\boldsymbol{2}}]$.
\end{proof} 
In next theorem, we give the complete energy ordering of those bicyclic sidigraphs in $\mathcal{D}_n^s$ whose sign of both cycles is positive and also give the complete energy ordering of those bicyclic sidigraphs in $\mathcal{D}_n^s$ whose one cycle is of positive sign and other cycle is of negative sign. For energy ordering of bicyclic sidigraphs in $\mathcal{D}_n^s$ with both cycles of positive sign, see Theorem $3.8$ \cite{XL2019}.
\begin{theorem}
Let $n>5$ and $r\in[2,n-2]$.
\begin{itemize}
	\item [$(i)$] If $n\equiv 2(\bmod~4)$ then we have the following energy ordering:
	\begin{itemize}
		\item [$(a)$] When $\frac{n}{2}-3 \equiv 2(\bmod~4)$:
			\begin{eqnarray*}
			&&E(D_n^s[2,{\boldsymbol{n-2}}]) > E(D_n^s[6,{\boldsymbol{n-6}}])> \dots > E(D_n^s[\frac{n-6}{2},{\boldsymbol{\frac{n+6}{2}}}])\\
			&>& E(D_n^s[\frac{n-2}{2},{\boldsymbol{\frac{n+2}{2}}}])> E(D_n^s[\frac{n-10}{2},{\boldsymbol{\frac{n+10}{2}}}]) >\dots> E(D_n^s[4,{\boldsymbol{n-4}}])\\
			&=& E(D_n^s[2,{\boldsymbol{n-4}}])> E(D_n^s[6,{\boldsymbol{n-8}}])>E(D_n^s[10,{\boldsymbol{n-12}}]) >\dots > E(\mathcal{D}_n^s[n-4,{\boldsymbol{2}}])\\
			&=& E(\mathcal{D}_n^s[{\boldsymbol{2}} ,n-4])> E(D_n^s[{\boldsymbol{2}}, n-8]) > E(D_n^s[{\boldsymbol{2}}, n-12]) >\\
			&&\dots> E(D_n^s[{\boldsymbol{2}},6])>E(D_n^s[{\boldsymbol{2}},2]) > E(D_n^s[{\boldsymbol{2}},{\boldsymbol{2}}]).  	
		\end{eqnarray*}
	Also
	\begin{eqnarray*}
		&&E(D_n^s[{\boldsymbol{4}},n-4]) > E(D_n^s[{\boldsymbol{8}},n-8])> \dots > E(D_n^s[{\boldsymbol{\frac{n-2}{2}}},\frac{n+2}{2}])\\
		&>& E(D_n^s[{\boldsymbol{\frac{n-6}{2}}},\frac{n+6}{2}])> E(D_n^s[{\boldsymbol{\frac{n-14}{2}}},\frac{n+14}{2}]) >\dots> E(D_n^s[{\boldsymbol{2}},n-2])\\
		&>&  E(D_n^s[{\boldsymbol{2}},n-6]) > E(D_n^s[{\boldsymbol{2}},n-10]) >\dots>  E(D_n^s[{\boldsymbol{2}},2]) >  E(D_n^s[{\boldsymbol{2}},{\boldsymbol{2}}]).	
	\end{eqnarray*}
\item [$(b)$] When $\frac{n}{2}-3 \equiv 0(\bmod~4)$:
\begin{eqnarray*}
&&E(D_n^s[2,{\boldsymbol{n-2}}]) > E(D_n^s[6,{\boldsymbol{n-6}}])> \dots > E(D_n^s[\frac{n-2}{2},{\boldsymbol{\frac{n+2}{2}}}])\\
&>& E(D_n^s[\frac{n-6}{2},{\boldsymbol{\frac{n+6}{2}}}])> E(D_n^s[\frac{n-14}{2},{\boldsymbol{\frac{n+14}{2}}}]) >\dots> E(D_n^s[4,{\boldsymbol{n-4}}])\\
&=& E(D_n^s[2,{\boldsymbol{n-4}}])> E(D_n^s[6,{\boldsymbol{n-8}}])>E(D_n^s[10,{\boldsymbol{n-12}}]) >\dots > E(D_n^s[n-4,{\boldsymbol{2}}])\\
&=&E(D_n^s[{\boldsymbol{2}},n-4])> E(D_n^s[{\boldsymbol{2}},n-8]) > E(D_n^s[{\boldsymbol{2}},n-12]) >\\
&&\dots> E(D_n^s[{\boldsymbol{2}},6])>E(D_n^s[{\boldsymbol{2}},2]) > E(D_n^s[{\boldsymbol{2}},{\boldsymbol{2}}]).
\end{eqnarray*}
Also
\begin{eqnarray*}
&&E(D_n^s[{\boldsymbol{4}},n-4]) > E(D_n^s[{\boldsymbol{8}},n-8])> \dots > E(D_n^s[{\boldsymbol{\frac{n-6}{2}}},\frac{n+6}{2}])\\
&>& E(D_n^s[{\boldsymbol{\frac{n-2}{2}}},\frac{n+2}{2}])> E(D_n^s[{\boldsymbol{\frac{n-10}{2}}},\frac{n+10}{2}]) >\dots> E(D_n^s[{\boldsymbol{2}},n-2])\\
&>&  E(D_n^s[{\boldsymbol{2}},n-6]) > E(D_n^s[{\boldsymbol{2}},n-10]) >\dots>  E(D_n^s[{\boldsymbol{2}},2]) >  E(D_n^s[{\boldsymbol{2}},{\boldsymbol{2}}]).
\end{eqnarray*}
	\end{itemize}
\item[$(ii)$] If $n\equiv 3(\bmod~4)$ then the following energy ordering holds:
\begin{itemize}
	\item [$(a)$] When $\frac{n-1}{2}-3 \equiv 2(\bmod~4)$:
	\begin{eqnarray*}
		&&E(D_n^s[2,{\boldsymbol{n-3}}]) > E(D_n^s[6,{\boldsymbol{n-7}}])> E(D_n^s[10,{\boldsymbol{n-11}}])> E(D_n^s[14,{\boldsymbol{n-15}}])\\ &&>\dots>E(D_n^s[\frac{n-7}{2},{\boldsymbol{\frac{n+5}{2}}}])
		> E(D_n^s[\frac{n-3}{2},{\boldsymbol{\frac{n+1}{2}}}])> E(D_n^s[\frac{n-11}{2},{\boldsymbol{\frac{n+9}{2}}}])\\
		 &&>E(D_n^s[\frac{n-19}{2},{\boldsymbol{\frac{n+17}{2}}}])>\dots> E(D_n^s[4,{\boldsymbol{n-5}}])\\
		&&= E(D_n^s[2,{\boldsymbol{n-5}}])> E(D_n^s[6,{\boldsymbol{n-9}}]) >E(D_n^s[10,{\boldsymbol{n-13}}])>\dots > E(D_n^s[n-5,{\boldsymbol{2}}])\\
		&&= E(D_n^s[{\boldsymbol{2}},n-5])>
	E(D_n^s[{\boldsymbol{2}},n-9]) >\dots> E(D_n^s[{\boldsymbol{2}},2]) > E(D_n^s[{\boldsymbol{2}},{\boldsymbol{2}}]).  	
	\end{eqnarray*}
	Also
	\begin{eqnarray*}
		&&E(D_n^s[{\boldsymbol{4}},n-5]) > E(D_n^s[{\boldsymbol{8}},n-9])> \dots > E(D_n^s[{\boldsymbol{\frac{n-7}{2}}},\frac{n+5}{2}])\\
		&>& E(D_n^s[{\boldsymbol{\frac{n-15}{2}}},\frac{n+13}{2}])> E(D_n^s[{\boldsymbol{\frac{n-23}{2}}},\frac{n+21}{2}])>\\
		&& E(D_n^s[{\boldsymbol{\frac{n-31}{2}}},\frac{n+29}{2}]) >\dots> E(D_n^s[{\boldsymbol{2}},n-3])\\
		&>&  E(D_n^s[{\boldsymbol{2}},n-7]) > E(D_n^s[{\boldsymbol{2}},n-11]) >\dots>  E(D_n^s[{\boldsymbol{2}},2]) >  E(D_n^s[{\boldsymbol{2}},{\boldsymbol{2}}]).	
	\end{eqnarray*} 
	\item [$(b)$] When $\frac{n-1}{2}-3 \equiv 0(\bmod~4)$:
	\begin{eqnarray*}
		&&E(D_n^s[2,{\boldsymbol{n-3}}]) > E(D_n^s[6,{\boldsymbol{n-7}}])> \dots > E(D_n^s[\frac{n-3}{2},{\boldsymbol{\frac{n+1}{2}}}])\\
		&>& E(D_n^s[\frac{n-7}{2},{\boldsymbol{\frac{n+5}{2}}}])> E(D_n^s[\frac{n-15}{2},{\boldsymbol{\frac{n+13}{2}}}])>\\
		&& E(D_n^s[\frac{n-23}{2},{\boldsymbol{\frac{n+21}{2}}}])>\dots> E(D_n^s[4,{\boldsymbol{n-5}}])\\
		&=& E(D_n^s[2,{\boldsymbol{n-5}}])> E(D_n^s[6,{\boldsymbol{n-9}}])>\dots > E(D_n^s[n-5,{\boldsymbol{2}}])\\
		&=& E(D_n^s[{\boldsymbol{2}},n-5]) > E(D_n^s[{\boldsymbol{2}},n-9]) > E(D_n^s[{\boldsymbol{2}},n-12]) >\dots> E(D_n^s[{\boldsymbol{2}},2]) > E(D_n^s[{\boldsymbol{2}},{\boldsymbol{2}}]).
	\end{eqnarray*}
	Also
	\begin{eqnarray*}
		&&E(D_n^s[{\boldsymbol{4}},n-5]) > E(D_n^s[{\boldsymbol{8}},n-9])> \dots > E(D_n^s[{\boldsymbol{\frac{n-7}{2}}},\frac{n+5}{2}])\\
		&>& E(D_n^s[{\boldsymbol{\frac{n-3}{2}}},\frac{n+1}{2}])> E(D_n^s[{\boldsymbol{\frac{n-11}{2}}},\frac{n+9}{2}])>\\
		&&E(D_n^s[{\boldsymbol{\frac{n-19}{2}}},\frac{n+17}{2}]) >\dots> E(D_n^s[{\boldsymbol{2}},n-3])\\
		&>&  E(D_n^s[{\boldsymbol{2}},n-7]) > E(D_n^s[{\boldsymbol{2}},n-11]) >\dots>  E(D_n^s[{\boldsymbol{2}},2]) >  E(D_n^s[{\boldsymbol{2}},{\boldsymbol{2}}]).
	\end{eqnarray*}

\end{itemize}
\item[$(iii)$] If $n\equiv 0(\bmod~4)$ then we have the following energy ordering:
\begin{itemize}
	\item [$(a)$] When $\frac{n}{2}\equiv 0(\bmod~4)$ and $r\equiv 0(\bmod~2)$:
\begin{eqnarray*}
	&&E(D_n^s[n-4,{\boldsymbol{4}}])> E(D_n^s[n-8,{\boldsymbol{8}}])>E(D_n^s[n-12,{\boldsymbol{12}}])>\dots> E(D_n^s[4,{\boldsymbol{n-4}}])\\
	&=&E(D_n^s[2,{\boldsymbol{n-4}}])>E(D_n^s[6,{\boldsymbol{n-8}}])>\dots>E(D_n^s[\frac{n-4}{2},{\boldsymbol{\frac{n}{2}}}])\\
	&>&E(D_n^s[\frac{n-8}{2},{\boldsymbol{\frac{n+4}{2}}}]) >E(D_n^s[\frac{n-16}{2},{\boldsymbol{\frac{n+12}{2}}}])>\dots>E(D_n^s[4,{\boldsymbol{n-6}}])\\
	&=&E(D_n^s[2,{\boldsymbol{n-6}}])> E(D_n^s[6,{\boldsymbol{n-10}}]) >\dots> E(D_n^s[n-6,{\boldsymbol{2}}])\\
	&>& E(D_n^s[n-10,{\boldsymbol{2}}])>\dots> E(D_n^s[2,{\boldsymbol{2}}])> E(D_n^s[{\boldsymbol{2}},{\boldsymbol{2}}]).
\end{eqnarray*}
Also
\begin{eqnarray*}
&& E(D_n^s[2,{\boldsymbol{n-2}}])> E(D_n^s[6,{\boldsymbol{n-6}}])> E(D_n^s[10,{\boldsymbol{n-10}}])> E(D_n^s[14,{\boldsymbol{n-14}}])\\
&&>\dots> E(D_n^s[n-6,{\boldsymbol{6}}])> E(D_n^s[n-2,{\boldsymbol{2}}])\\
&&= E(D_n^s[{\boldsymbol{2}},n-2])> E(D_n^s[{\boldsymbol{2}},n-6]) > E(D_n^s[{\boldsymbol{2}},n-10]) >\dots > E(D_n^s[{\boldsymbol{2}},2])> E(D_n^s[{\boldsymbol{2}},{\boldsymbol{2}}]),
\end{eqnarray*}
and
\begin{eqnarray*}
&&E(D_n^s[{\boldsymbol{4}}, {\boldsymbol{n-4}}]) > E(D_n^s[{\boldsymbol{8}}, {\boldsymbol{n-8}}])> \dots> E(D_n^s[{\boldsymbol{\frac{n}{2}}}, {\boldsymbol{\frac{n}{2}}}])\\
&>& E(D_n^s[{\boldsymbol{\frac{n-4}{2}}}, {\boldsymbol{\frac{n+4}{2}}}])> E(D_n^s[{\boldsymbol{\frac{n-12}{2}}}, {\boldsymbol{\frac{n+12}{2}}}])> \dots> 	E(D_n^s[{\boldsymbol{2}}, {\boldsymbol{n-2}}])\\
&>& E(D_n^s[{\boldsymbol{2}}, {\boldsymbol{n-6}}]) >E(D_n^s[{\boldsymbol{2}}, {\boldsymbol{n-10}}])>\dots> E(D_n^s[{\boldsymbol{2}}, {\boldsymbol{2}}]).  
\end{eqnarray*}
\item[$(b)$] If $\frac{n}{2}\equiv 2(\bmod~2)$ and $r\equiv 0(\bmod~2)$:
\begin{eqnarray*}
	&&E(D_n^s[m-4,{\boldsymbol{4}}])> E(D_n^s[n-8,{\boldsymbol{8}}])>\dots> E(D_n^s[8,{\boldsymbol{n-8}}])> E(D_n^s[4,{\boldsymbol{n-4}}])\\
	&=& E(D_n^s[2,{\boldsymbol{n-4}}]) > E(D_n^s[6,{\boldsymbol{n-8}}]) >\dots > E(D_n^s[\frac{n-8}{2},{\boldsymbol{\frac{n+4}{2}}}])\\
	&>& E(D_n^s[\frac{n-4}{2},{\boldsymbol{\frac{n}{2}}}])> E(D_n^s[\frac{n-12}{2},{\boldsymbol{\frac{n+8}{2}}}])>\dots> E(D_n^s[4,{\boldsymbol{n-6}}])\\
	&=&E(D_n^s[2,{\boldsymbol{n-6}}]) > E(D_n^s[6,{\boldsymbol{n-10}}]) > \dots> E(D_n^s[n-6,{\boldsymbol{2}}])\\
	&>&E(D_n^s[n-10,{\boldsymbol{2}}])>\dots> E(D_n^s[2,{\boldsymbol{2}}])>E(D_n^s[{\boldsymbol{2}},{\boldsymbol{2}}]).
\end{eqnarray*}
Also
\begin{eqnarray*}
&& E(D_n^s[2,{\boldsymbol{n-2}}])> E(D_n^s[6,{\boldsymbol{n-6}}])>E(D_n^s[10,{\boldsymbol{n-10}}])> E(D_n^s[14,{\boldsymbol{n-14}}])\\
&&>\dots> E(D_n^s[n-6,{\boldsymbol{6}}])> E(D_n^s[n-2,{\boldsymbol{2}}])\\
&&=E(D_n^s[{\boldsymbol{2}},n-2])> E(D_n^s[{\boldsymbol{2}},n-6]) > E(D_n^s[{\boldsymbol{2}},n-10]) >\dots > E(D_n^s[{\boldsymbol{2}},2])> E(D_n^s[{\boldsymbol{2}},{\boldsymbol{2}}]),
\end{eqnarray*}
and
\begin{eqnarray*}
&&E(D_n^s[{\boldsymbol{4}}, {\boldsymbol{n-4}}]) > E(D_n^s[{\boldsymbol{8}}, {\boldsymbol{n-8}}])> \dots> E(D_n^s[{\boldsymbol{\frac{n-4}{2}}}, {\boldsymbol{\frac{n+4}{2}}}])\\
&>& E(D_n^s[{\boldsymbol{\frac{n}{2}}}, {\boldsymbol{\frac{n}{2}}}])> E(D_n^s[{\boldsymbol{\frac{n-8}{2}}}, {\boldsymbol{\frac{n+8}{2}}}])> \dots> 	E(D_n^s[{\boldsymbol{2}}, {\boldsymbol{n-2}}]) \\
&>& E(D_n^s[{\boldsymbol{2}}, {\boldsymbol{n-6}}]) >E(D_n^s[{\boldsymbol{2}}, {\boldsymbol{n-10}}])>\dots> E(D_n^s[{\boldsymbol{2}}, {\boldsymbol{2}}]). 
\end{eqnarray*}
\end{itemize}
\item[$(iv)$] If $n\equiv 1(\bmod~4)$ then the following energy ordering holds:
\begin{itemize}
\item [$(a)$] When $\frac{n-1}{2}\equiv 0(\bmod~4)$ and $r\equiv 0(\bmod~2)$:
\begin{eqnarray*}
	&&E(D_n^s[n-5,{\boldsymbol{4}}])> E(D_n^s[n-9,{\boldsymbol{8}}])>E(D_n^s[n-13,{\boldsymbol{12}}])>\dots> E(D_n^s[4,{\boldsymbol{n-5}}])\\
	&=&E(D_n^s[2,{\boldsymbol{n-5}}])>E(D_n^s[6,{\boldsymbol{n-9}}])>\dots>E(D_n^s[\frac{n-5}{2},{\boldsymbol{\frac{n-1}{2}}}])\\
	&>&E(D_n^s[\frac{n-9}{2},{\boldsymbol{\frac{n+3}{2}}}]) >E(D_n^s[\frac{n-17}{2},{\boldsymbol{\frac{n+11}{2}}}])>\dots>E(D_n^s[4,{\boldsymbol{n-7}}])\\
	&=&E(D_n^s[2,{\boldsymbol{n-7}}])> E(D_n^s[6,{\boldsymbol{n-11}}]) >\dots> E(D_n^s[n-7,{\boldsymbol{2}}])\\
	&>& E(D_n^s[n-11,{\boldsymbol{2}}])>\dots> E(D_n^s[2,{\boldsymbol{2}}])> E(D_n^s[{\boldsymbol{2}},{\boldsymbol{2}}]).
\end{eqnarray*}
Also
\begin{eqnarray*}
&& E(D_n^s[2,{\boldsymbol{n-3}}])> E(D_n^s[6,{\boldsymbol{n-7}}])> E(D_n^s[10,{\boldsymbol{n-11}}])> E(D_n^s[14,{\boldsymbol{n-15}}])\\
&&>\dots> E(D_n^s[n-7,{\boldsymbol{6}}])> E(D_n^s[n-3,{\boldsymbol{2}}])\\
&&> E(D_n^s[n-7,{\boldsymbol{2}}]) > E(D_n^s[n-11,{\boldsymbol{2}}]) >\dots > E(D_n^s[2,{\boldsymbol{2}}])> E(D_n^s[{\boldsymbol{2}},{\boldsymbol{2}}]),
\end{eqnarray*}
and
\begin{eqnarray*}
	&&E(D_n^s[{\boldsymbol{4}}, {\boldsymbol{n-5}}]) > E(D_n^s[{\boldsymbol{8}}, {\boldsymbol{n-9}}])> \dots> E(D_n^s[{\boldsymbol{\frac{n-1}{2}}}, {\boldsymbol{\frac{n-1}{2}}}])\\
	&>& E(D_n^s[{\boldsymbol{\frac{n-5}{2}}}, {\boldsymbol{\frac{n+3}{2}}}])> E(D_n^s[{\boldsymbol{\frac{n-13}{2}}}, {\boldsymbol{\frac{n+11}{2}}}])> \dots\\
	&>& E(D_n^s[{\boldsymbol{2}}, {\boldsymbol{n-3}}])>E(D_n^s[{\boldsymbol{2}}, {\boldsymbol{n-7}}]) >E(D_n^s[{\boldsymbol{2}}, {\boldsymbol{n-11}}])>\dots> E(D_n^s[{\boldsymbol{2}}, {\boldsymbol{2}}]).  
\end{eqnarray*}
	\item[$(b)$] When $\frac{n-1}{2}\equiv 2(\bmod~2)$ and $r\equiv 0(\bmod~2)$:
	\begin{eqnarray*}
		&&E(D_n^s[n-5,{\boldsymbol{4}}])> E(D_n^s[n-9,{\boldsymbol{8}}])>\dots> E(D_n^s[8,{\boldsymbol{n-9}}])> E(D_n^s[4,{\boldsymbol{n-5}}])\\
		&=& E(D_n^s[2,{\boldsymbol{n-5}}]) > E(D_n^s[6,{\boldsymbol{n-9}}]) >\dots > E(D_n^s[\frac{n-9}{2},{\boldsymbol{\frac{n+3}{2}}}])\\
		&>& E(D_n^s[\frac{n-5}{2},{\boldsymbol{\frac{n-1}{2}}}])> E(D_n^s[\frac{n-13}{2},{\boldsymbol{\frac{n+7}{2}}}])>\dots> E(D_n^s[4,{\boldsymbol{n-7}}])\\
		&=&E(D_n^s[2,{\boldsymbol{n-7}}]) > E(D_n^s[6,{\boldsymbol{n-11}}]) > \dots> E(D_n^s[n-7,{\boldsymbol{2}}])\\
		&>&E(D_n^s[n-10,{\boldsymbol{2}}])>\dots> E(D_n^s[2,{\boldsymbol{2}}])>E(D_n^s[{\boldsymbol{2}},{\boldsymbol{2}}]).
	\end{eqnarray*}
	Also
	\begin{eqnarray*}
		&& E(D_n^s[2,{\boldsymbol{n-3}}])> E(D_n^s[6,{\boldsymbol{n-7}}])> E(D_n^s[10,{\boldsymbol{n-11}}])>E(D_n^s[14,{\boldsymbol{n-15}}])\\
		&&>\dots> E(D_n^s[n-7,{\boldsymbol{6}}])> E(D_n^s[n-3,{\boldsymbol{2}}])\\
		&&> E(D_n^s[n-7,{\boldsymbol{2}}]) > E(D_n^s[n-11,{\boldsymbol{2}}]) >\dots > E(D_n^s[2,{\boldsymbol{2}}])> E(D_n^s[{\boldsymbol{2}},{\boldsymbol{2}}]),
	\end{eqnarray*}
and
\begin{eqnarray*}
	&&E(D_n^s[{\boldsymbol{4}}, {\boldsymbol{n-5}}]) > E(D_n^s[{\boldsymbol{8}}, {\boldsymbol{n-9}}])> \dots> E(D_n^s[{\boldsymbol{\frac{n-5}{2}}}, {\boldsymbol{\frac{n+3}{2}}}])\\
	&>& E(D_n^s[{\boldsymbol{\frac{n-1}{2}}}, {\boldsymbol{\frac{n-1}{2}}}])> E(D_n^s[{\boldsymbol{\frac{n-9}{2}}}, {\boldsymbol{\frac{n+7}{2}}}])> \dots> 	E(D_n^s[{\boldsymbol{2}}, {\boldsymbol{n-3}}]) \\
	&>& E(D_n^s[{\boldsymbol{2}}, {\boldsymbol{n-7}}]) >E(D_n^s[{\boldsymbol{2}}, {\boldsymbol{n-11}}])>\dots> E(D_n^s[{\boldsymbol{2}}, {\boldsymbol{2}}]). 
\end{eqnarray*}
\end{itemize}
\end{itemize} 
\end{theorem}
\begin{proof}
As we know $\cot\frac{\pi}{4} = \csc\frac{\pi}{2}$
and $\csc z$ and $\cot z$ is decreasing for $z\in(0,\frac{\pi}{2}]$. Using this fact and above two equations, we get the required energy ordering of bicyclic sidigraphs in $\mathcal{D}_n^s$.
\end{proof}
\subsection{Both cycles of odd length}
In this section, we find energy ordering of those bicyclic sidigraphs in $\mathcal{D}_n^s$ that contain cycles of odd length. Note that for $r\equiv 1(\bmod~2)$, $E(C_r) = E(\boldsymbol{C}_r)$. Hence we only consider the case when both cycles are positive. The proof are similar to the proofs of Section 3.1 and thus omitted.\vspace{.1cm}

\begin{lemma}\label{lemn32}
Let $n>5$ and $n\equiv0(\bmod~4)$. Take $r\in[2,n-2]$ satisfying $r\equiv1(\bmod~2)$ and $n-r\equiv1(\bmod~2)$. Then $E(D_n^s[r,n-r])$ has maximum value at $r=3$. Therefore the following energy ordering holds:
\begin{equation*}
E(D_n^s[3,n-3]) > E(D_n^s[5,n-5]) > \dots > E(D_n^s[\frac{n-2}{2},\frac{n+2}{2}]).
\end{equation*}
\end{lemma}
\begin{lemma}\label{lemn33}
	Let $n>5$ and $n\equiv2(\bmod~4)$. Take $r\in[2,n-2]$ satisfying $r\equiv1(\bmod~2)$ and $n-r\equiv1(\bmod~2)$. Then $E(D_n^s[r,n-r])$ has maximum value at $r=3$. Therefore the following energy ordering holds:
	\begin{equation*}
	E(D_n^s[3,n-3]) > E(D_n^s[5,n-5]) > \dots > E(D_n^s[\frac{n}{2},\frac{n}{2}]).
	\end{equation*}
\end{lemma}
By changing $n$ in Lemma \ref{lemn32} to $n-1$, we get the following result.
\begin{lemma}\label{lemn34}
Let $n>5$ and $n\equiv1(\bmod~4)$. Take $r\in[2,n-2]$ satisfying $r\equiv1(\bmod~2)$ and $n-r-1\equiv1(\bmod~2)$. Then $E(D_n^s[r,n-r-1])$ has maximum value at $r=3$. Therefore the following energy ordering holds:
\begin{equation*}
E(D_n^s[3,n-4]) > E(D_n^s[5,n-6]) > \dots > E(D_n^s[\frac{n-3}{2},\frac{n+1}{2}]).
\end{equation*}
\end{lemma}
By changing $n$ in Lemma \ref{lemn33} to $n-1$, the following result is obtained.
\begin{lemma}\label{lemn35}
	Let $n>5$ and $n\equiv3(\bmod~4)$. Take $r\in[2,n-2]$ satisfying $r\equiv1(\bmod~2)$ and $n-r-1\equiv1(\bmod~2)$. Then $E(D_n^s[r,n-r-1])$ has maximum value at $r=3$. Therefore the following energy ordering holds:
	\begin{equation*}
	E(D_n^s[3,n-4]) > E(D_n^s[5,n-6]) > \dots > E(D_n^s[\frac{n-1}{2},\frac{n-1}{2}]).
	\end{equation*}
\end{lemma}
Combining Lemmas \ref{lemn32} $\sim$ \ref{lemn35}, the following corollary is obtained.
\begin{corollary}\label{cor31}
	Suppose $r\equiv1(\bmod~2)$ and $n> 5$.
	\begin{itemize}
		\item[$(i)$] If $n\equiv0(\bmod~2)$ then
		$E(D_n^s[3, n-3]) \geq E(D_n^s[r, n-r]).$ 
		\item[$(ii)$] If $n\equiv1(\bmod~2)$ then
		$E(D_n^s[3, n-4]) \geq E(D_n^s[r, n-r-1]).$ 
	\end{itemize}
\end{corollary}
 Now we give the extremal energy of bicyclic sidigraphs in the class $\mathcal{D}_n^s$.
\begin{theorem}
	Let $S\in\mathcal{D}_n^s$ be a sidigraph with odd directed cycles.
	\begin{itemize}
		\item[$(i)$] For $n\equiv0(\bmod~2)$, the maximal energy of $S$ is attained if $S\cong D_n^s[3,n-3]$.
		\item[$(ii)$] For $n\equiv1(\bmod~2)$, the maximal energy of $S$ is attained if $S\cong D_n^s[3,n-4]$. 
		\item[$(iii)$] The minimal energy of $S$ is attained if $S\cong D_n^s[3,3]$.
	\end{itemize}
\end{theorem}
\begin{proof}
The proof of Part $(i)$ and $(ii)$ follows from Corollary \ref{cor31}. \\
$(iii)$. As for odd integers $r_1$ and $r_2$ with $r_1 \geq r_2 \geq 3$, it holds that $E(C_{r_1}) \geq E(C_{r_2})$. Hence the minimal energy of $S$ is attained if $S\cong D_n^s[3,3]$.
\end{proof}
In next theorem, we give the complete energy ordering of those bicyclic sidigraphs in $\mathcal{D}_n^s$ whose both cycles are of odd length.
\begin{theorem}
Let $n>5$ and $r\in[2,n-2]$.
\begin{itemize}
\item [$(i)$] If $n\equiv 0(\bmod~4)$ then we have the following energy ordering:
\begin{eqnarray*}
&&E(D_n^s[3,n-3]) > E(D_n^s[5,n-5]) > \dots > E(D_n^s[\frac{n-2}{2},\frac{n+2}{2}])
> E(D_n^s[\frac{n-6}{2},\frac{n+2}{2}])\\
&&> \dots> E(D_n^s[3,\frac{n+2}{2}])
> E(D_n^s[3,\frac{n-2}{2}])>E(D_n^s[3,\frac{n-6}{2}])>\dots>E(D_n^s[3,3]).
\end{eqnarray*}	
\item [$(ii)$] If $n\equiv 1(\bmod~4)$ then we have the following energy ordering:
\begin{eqnarray*}
	&&E(D_n^s[3,n-4]) > E(D_n^s[5,n-6]) > \dots > E(D_n^s[\frac{n-3}{2},\frac{n+1}{2}])
	> E(D_n^s[\frac{n-7}{2},\frac{n+1}{2}])\\
	&& > \dots> E(D_n^s[3,\frac{n+1}{2}])
	>E(D_n^s[3,\frac{n-3}{2}])>E(D_n^s[3,\frac{n-7}{2}])>\dots>E(D_n^s[3,3]).
\end{eqnarray*}
\item [$(iii)$] If $n\equiv 2(\bmod~4)$ then we have the following energy ordering:
\begin{eqnarray*}
&&E(D_n^s[3,n-3]) > E(D_n^s[5,n-5]) > \dots > E(D_n^s[\frac{n}{2},\frac{n}{2}])
> E(D_n^s[\frac{n-4}{2},\frac{n}{2}])\\ &&> \dots> E(D_n^s[3,\frac{n}{2}])
> E(D_n^s[3,\frac{n-4}{2}])>E(D_n^s[3,\frac{n-8}{2}])>\dots>E(D_n^s[3,3]).
\end{eqnarray*}
\item [$(iv)$] If $n\equiv 3(\bmod~4)$ then we have the following energy ordering:
\begin{eqnarray*}
	&&E(D_n^s[3,n-4]) > E(D_n^s[5,n-6]) > \dots > E(D_n^s[\frac{n-1}{2},\frac{n-1}{2}])
	> E(D_n^s[\frac{n-5}{2},\frac{n-1}{2}])\\ 
	&&> \dots> E(D_n^s[3,\frac{n-1}{2}])> E(D_n^s[3,\frac{n-5}{2}])>E(D_n^s[3,\frac{n-9}{2}])>\dots>E(D_n^s[3,3]).
\end{eqnarray*}		
\end{itemize}
\end{theorem}
\subsection{One cycle of odd length and one cycle of even length}
In this section, we find energy ordering of those bicyclic sidigraphs in $\mathcal{D}_n^s$ whose one cycle is of even length and one is of odd length. For $n\equiv0(\bmod~2)$, if $r\equiv 1(\bmod~2)$ then $n-r \equiv 1 (\bmod~2)$ and if $r\equiv 0(\bmod~2)$ then $n-r \equiv 0(\bmod~2)$. So we only consider the case when $n \equiv 1(\bmod~2)$. Note that for $r\equiv 0(\bmod~2)$ and $n-r \equiv1(\bmod~2)$ then $ E(D_n^s[r,n-r]) = E(D_n^s[r,\boldsymbol{n-r}])$. Hence we only have to give the energy ordering of those bicyclic sidigraphs in $\mathcal{D}_n^s$ whose both cycles are positive or both cycles are negative. The proof are similar to the proofs of Section 3.1 and thus omitted.\vspace{.1cm}

Now we give the energy ordering of those bicyclic sidigraphs in $\mathcal{D}_n^s$ whose both cycles are positive.
\begin{lemma}\label{lemn311}
Suppose $n>5$, $n\equiv0(\bmod~3)$ and $r\in[2,n-2]$ satisfying $r\equiv 0(\bmod~2)$ and $n-r \equiv 1(\bmod~2)$. Then we have the following energy ordering:
\begin{itemize}
	\item [$(i)$] Let $r \equiv 2(\bmod~4)$.
	\begin{itemize}
		\item [$(a)$] If $r \in [2,\frac{2n}{3}]$ then
		\begin{equation*}
		E(D_n^s[2,n-2])>E(D_n^s[6,n-6]) > \dots > E(D_n^s[\frac{2n}{3},\frac{n}{3}]).
		\end{equation*}
	\item [$(b)$] If $r \in [\frac{2n}{3},n-2]$ and $n-3\equiv 2(\bmod~4)$ then
	\begin{equation*}
	E(D_n^s[n-3,3])>E(D_n^s[n-7,7]) > \dots > E(D_n^s[\frac{2n}{3},\frac{n}{3}]).
	\end{equation*}	
	\item [$(c)$] If $r \in [\frac{2n}{3},n-2]$ and $n-3\equiv 0(\bmod~4)$ then
	\begin{equation*}
	E(D_n^s[n-5,5])>E(D_n^s[n-9,9]) > \dots > E(D_n^s[\frac{2n}{3},\frac{n}{3}]).
	\end{equation*}	
	\end{itemize}
\item[$(ii)$] Let $r \equiv 0(\bmod~4)$.
\begin{itemize}
	\item [$(a)$] If $n-3\equiv 2(\bmod~4)$ then
	\begin{equation*}
	E(D_n^s[n-5,5])>E(D_n^s[n-9,9]) > \dots > E(D_n^s[4,n-4]).
	\end{equation*}
	\item [$(b)$] If $n-3\equiv 0(\bmod~4)$ then
	\begin{equation*}
	E(D_n^s[n-3,5])>E(D_n^s[n-7,9]) > \dots > E(D_n^s[4,n-4]).
	\end{equation*}
\end{itemize}
\end{itemize}
\end{lemma}
\begin{lemma}\label{lemn312}
	Suppose $n>5$, $n\equiv1(\bmod~3)$ and $r\in[2,n-2]$ satisfying $r\equiv 0(\bmod~2)$ and $n-r \equiv 1(\bmod~2)$. Then we have the following energy ordering:
	\begin{itemize}
		\item [$(i)$] Let $r \equiv 2(\bmod~4)$.
		\begin{itemize}
			\item [$(a)$] If $r \in [2,\frac{2n}{3}]$ then
			\begin{equation*}
			E(D_n^s[2,n-2])>E(D_n^s[6,n-6]) > \dots > E(D_n^s[\frac{2n-8}{3},\frac{n+8}{3}]).
			\end{equation*}
			\item [$(b)$] If $r \in [\frac{2n}{3},n-2]$ and $n-3\equiv 2(\bmod~4)$ then
			\begin{equation*}
			E(D_n^s[n-3,3])>E(D_n^s[n-7,7]) > \dots > E(D_n^s[\frac{2n+4}{3},\frac{n-4}{3}]).
			\end{equation*}	
			\item [$(c)$] If $r \in [\frac{2n}{3},n-2]$ and $n-3\equiv 0(\bmod~4)$ then
			\begin{equation*}
			E(D_n^s[n-5,5])>E(D_n^s[n-9,9]) > \dots > E(D_n^s[\frac{2n+4}{3},\frac{n-4}{3}]).
			\end{equation*}	
		\end{itemize}
		\item[$(ii)$] Let $r \equiv 0(\bmod~4)$.
		\begin{itemize}
			\item [$(a)$] If $n-3\equiv 2(\bmod~4)$ then
			\begin{equation*}
			E(D_n^s[n-5,5])>E(D_n^s[n-9,9]) > \dots > E(D_n^s[4,n-4]).
			\end{equation*}
			\item [$(b)$] If $n-3\equiv 0(\bmod~4)$ then
			\begin{equation*}
			E(D_n^s[n-3,5])>E(D_n^s[n-7,9]) > \dots > E(D_n^s[4,n-4]).
			\end{equation*}
		\end{itemize}
	\end{itemize}
\end{lemma}
\begin{lemma}\label{lemn313}
	Suppose $n>5$, $n\equiv2(\bmod~3)$ and $r\in[2,n-2]$ satisfying $r\equiv 0(\bmod~2)$ and $n-r \equiv 1(\bmod~2)$. Then we have the following energy ordering:
	\begin{itemize}
		\item [$(i)$] Let $r \equiv 2(\bmod~4)$.
		\begin{itemize}
			\item [$(a)$] If $r \in [2,\frac{2n}{3}]$ then
			\begin{equation*}
			E(D_n^s[2,n-2])>E(D_n^s[6,n-6]) > \dots > E(D_n^s[\frac{2n-4}{3},\frac{n+4}{3}]).
			\end{equation*}
			\item [$(b)$] If $r \in [\frac{2n}{3},n-2]$ and $n-3\equiv 2(\bmod~4)$ then
			\begin{equation*}
			E(D_n^s[n-3,3])>E(D_n^s[n-7,7]) > \dots > E(D_n^s[\frac{2n+8}{3},\frac{n-8}{3}]).
			\end{equation*}	
			\item [$(c)$] If $r \in [\frac{2n}{3},n-2]$ and $n-3\equiv 0(\bmod~4)$ then
			\begin{equation*}
			E(D_n^s[n-5,5])>E(D_n^s[n-9,9]) > \dots > E(D_n^s[\frac{2n+8}{3},\frac{n-8}{3}]).
			\end{equation*}	
		\end{itemize}
		\item[$(ii)$] Let $r \equiv 0(\bmod~4)$.
		\begin{itemize}
			\item [$(a)$] If $n-3\equiv 2(\bmod~4)$ then
			\begin{equation*}
			E(D_n^s[n-5,5])>E(D_n^s[n-9,9]) > \dots > E(D_n^s[4,n-4]).
			\end{equation*}
			\item [$(b)$] If $n-3\equiv 0(\bmod~4)$ then
			\begin{equation*}
			E(D_n^s[n-3,5])>E(D_n^s[n-7,9]) > \dots > E(D_n^s[4,n-4]).
			\end{equation*}
		\end{itemize}
	\end{itemize}
\end{lemma}
Now we give the energy ordering of those bicyclic sidigraphs in $\mathcal{D}_n^s$ whose both cycles are negative.
\begin{lemma}\label{lemn314}
	Suppose $n>5$, $n\equiv0(\bmod~3)$ and $r\in[2,n-2]$ satisfying $r\equiv 0(\bmod~2)$ and $n-r \equiv 1(\bmod~2)$. Then we have the following energy ordering:
	\begin{itemize}
		\item [$(i)$] Let $r \equiv 0(\bmod~4)$.
		\begin{itemize}
			\item [$(a)$] If $r \in [2,\frac{2n}{3}]$ then
			\begin{equation*}
			E(D_n^s[\boldsymbol{4},\boldsymbol{n-4}])>E(D_n^s[\boldsymbol{8},\boldsymbol{n-8}]) > \dots > E(D_n^s[\boldsymbol{\frac{2n-6}{3}},\boldsymbol{\frac{n+6}{3}}]).
			\end{equation*}
			\item [$(b)$] If $r \in [\frac{2n}{3},n-2]$ and $n-3\equiv 2(\bmod~4)$ then
			\begin{equation*}
			E(D_n^s[\boldsymbol{n-5},\boldsymbol{5}])>E(D_n^s[\boldsymbol{n-9},\boldsymbol{9}]) > \dots > E(D_n^s[\boldsymbol{\frac{2n+6}{3}},\boldsymbol{\frac{n-6}{3}}]).
			\end{equation*}	
			\item [$(c)$] If $r \in [\frac{2n}{3},n-2]$ and $n-3\equiv 0(\bmod~4)$ then
			\begin{equation*}
			E(D_n^s[\boldsymbol{n-3},\boldsymbol{3}])>E(D_n^s[\boldsymbol{n-7},\boldsymbol{7}]) > \dots > E(D_n^s[\boldsymbol{\frac{2n+6}{3}},\boldsymbol{\frac{n-6}{3}}]).
			\end{equation*}	
		\end{itemize}
		\item[$(ii)$] Let $r \equiv 2(\bmod~4)$.
		\begin{itemize}
			\item [$(a)$] If $n-3\equiv 2(\bmod~4)$ then
			\begin{equation*}
				E(D_n^s[\boldsymbol{n-3},\boldsymbol{3}])>	E(D_n^s[\boldsymbol{n-7},\boldsymbol{7}]) > \dots > 	E(D_n^s[\boldsymbol{2},\boldsymbol{n-2}]).
			\end{equation*}
			\item [$(b)$] If $n-3\equiv 0(\bmod~4)$ then
			\begin{equation*}
			E(D_n^s[\boldsymbol{n-5},\boldsymbol{5}])>	E(D_n^s[\boldsymbol{n-9},\boldsymbol{9}]) > \dots > 	E(D_n^s[\boldsymbol{2},\boldsymbol{n-2}]).
			\end{equation*}
		\end{itemize}
	\end{itemize}
\end{lemma}

\begin{lemma}\label{lemn315}
	Suppose $n>5$, $n\equiv1(\bmod~3)$ and $r\in[2,n-2]$ satisfying $r\equiv 0(\bmod~2)$ and $n-r \equiv 1(\bmod~2)$. Then we have the following energy ordering:
	\begin{itemize}
		\item [$(i)$] Let $r \equiv 0(\bmod~4)$.
		\begin{itemize}
			\item [$(a)$] If $r \in [2,\frac{2n}{3}]$ then
			\begin{equation*}
			E(D_n^s[\boldsymbol{4},\boldsymbol{n-4}])>E(D_n^s[\boldsymbol{8},\boldsymbol{n-8}]) > \dots > E(D_n^s[\boldsymbol{\frac{2n-2}{3}},\boldsymbol{\frac{n+2}{3}}]).
			\end{equation*}
			\item [$(b)$] If $r \in [\frac{2n}{3},n-2]$ and $n-3\equiv 2(\bmod~4)$ then
			\begin{equation*}
			E(D_n^s[\boldsymbol{n-5},\boldsymbol{5}])>E(D_n^s[\boldsymbol{n-9},\boldsymbol{9}]) > \dots > E(D_n^s[\boldsymbol{\frac{2n+10}{3}},\boldsymbol{\frac{n-10}{3}}]).
			\end{equation*}	
			\item [$(c)$] If $r \in [\frac{2n}{3},n-2]$ and $n-3\equiv 0(\bmod~4)$ then
			\begin{equation*}
			E(D_n^s[\boldsymbol{n-3},\boldsymbol{3}])>E(D_n^s[\boldsymbol{n-7},\boldsymbol{7}]) > \dots > E(D_n^s[\boldsymbol{\frac{2n+10}{3}},\boldsymbol{\frac{n-10}{3}}]).
			\end{equation*}	
		\end{itemize}
		\item[$(ii)$] Let $r \equiv 2(\bmod~4)$.
		\begin{itemize}
			\item [$(a)$] If $n-3\equiv 2(\bmod~4)$ then
			\begin{equation*}
			E(D_n^s[\boldsymbol{n-3},\boldsymbol{3}])>	E(D_n^s[\boldsymbol{n-7},\boldsymbol{7}]) > \dots > 	E(D_n^s[\boldsymbol{2},\boldsymbol{n-2}]).
			\end{equation*}
			\item [$(b)$] If $n-3\equiv 0(\bmod~4)$ then
			\begin{equation*}
			E(D_n^s[\boldsymbol{n-5},\boldsymbol{5}])>	E(D_n^s[\boldsymbol{n-9},\boldsymbol{9}]) > \dots > 	E(D_n^s[\boldsymbol{2},\boldsymbol{n-2}]).
			\end{equation*}
		\end{itemize}
	\end{itemize}
\end{lemma}
\begin{lemma}\label{lemn316}
	Suppose $n>5$, $n\equiv2(\bmod~3)$ and $r\in[2,n-2]$ satisfying $r\equiv 0(\bmod~2)$ and $n-r \equiv 1(\bmod~2)$. Then we have the following energy ordering:
	\begin{itemize}
		\item [$(i)$] Let $r \equiv 0(\bmod~4)$.
		\begin{itemize}
			\item [$(a)$] If $r \in [2,\frac{2n}{3}]$ then
			\begin{equation*}
			E(D_n^s[\boldsymbol{4},\boldsymbol{n-4}])>E(D_n^s[\boldsymbol{8},\boldsymbol{n-8}]) > \dots > E(D_n^s[\boldsymbol{\frac{2n-10}{3}},\boldsymbol{\frac{n+10}{3}}]).
			\end{equation*}
			\item [$(b)$] If $r \in [\frac{2n}{3},n-2]$ and $n-3\equiv 2(\bmod~4)$ then
			\begin{equation*}
			E(D_n^s[\boldsymbol{n-5},\boldsymbol{5}])>E(D_n^s[\boldsymbol{n-9},\boldsymbol{9}]) > \dots > E(D_n^s[\boldsymbol{\frac{2n+2}{3}},\boldsymbol{\frac{n-2}{3}}]).
			\end{equation*}	
			\item [$(c)$] If $r \in [\frac{2n}{3},n-2]$ and $n-3\equiv 0(\bmod~4)$ then
			\begin{equation*}
			E(D_n^s[\boldsymbol{n-3},\boldsymbol{3}])>E(D_n^s[\boldsymbol{n-7},\boldsymbol{7}]) > \dots > E(D_n^s[\boldsymbol{\frac{2n+2}{3}},\boldsymbol{\frac{n-2}{3}}]).
			\end{equation*}	
		\end{itemize}
		\item[$(ii)$] Let $r \equiv 2(\bmod~4)$.
		\begin{itemize}
			\item [$(a)$] If $n-3\equiv 2(\bmod~4)$ then
			\begin{equation*}
			E(D_n^s[\boldsymbol{n-3},\boldsymbol{3}])>	E(D_n^s[\boldsymbol{n-7},\boldsymbol{7}]) > \dots > 	E(D_n^s[\boldsymbol{2},\boldsymbol{n-2}]).
			\end{equation*}
			\item [$(b)$] If $n-3\equiv 0(\bmod~4)$ then
			\begin{equation*}
			E(D_n^s[\boldsymbol{n-5},\boldsymbol{5}])>	E(D_n^s[\boldsymbol{n-9},\boldsymbol{9}]) > \dots > 	E(D_n^s[\boldsymbol{2},\boldsymbol{n-2}]).
			\end{equation*}
		\end{itemize}
	\end{itemize}
\end{lemma}
Now we give the extremal energy of bicyclic sidigraphs in the class $\mathcal{D}_n^s$.
\begin{theorem}
	Let $S\in\mathcal{D}_n^s$ be a sidigraph with one cycle of even length and one of odd length.
	\begin{itemize}
		\item[$(i)$] For $n\equiv1(\bmod~2)$, the maximal energy of $S$ is attained if $S\cong D_n^s[2,n-2]$.
		\item[$(iii)$] The minimal energy of $S$ is attained if $S\cong D_n^s[\boldsymbol{2},3]$.
	\end{itemize}
\end{theorem}
\begin{proof}
$(i).$ For proof, see Theorem 7 \cite{KF}.\\
$(ii).$ Since for odd integers $r_1$ and $r_2$ with $r_1 \geq r_2 \geq 3$, it holds that $E(C_{r_1}) \geq E(C_{r_2})$ and $E(\boldsymbol{C}_2)=0$. Hence the minimal energy of $S$ is attained if $S\cong D_n^s[\boldsymbol{2},3]$.
\end{proof}
In next theorem, we give the complete energy ordering of those bicyclic sidigraphs in $\mathcal{D}_n^s$ whose one cycle is of even length and one is of odd length.
\begin{theorem}
	Let $n$ is odd with $n>5$ and $r\in[2,n-2]$.
	\begin{itemize}
		\item [$(1)$] If $n\equiv 0(\bmod~3)$ then we have the following energy ordering:
		\begin{itemize}
			\item [$(i)$] Let $r \equiv 2(\bmod~4)$.
			\begin{itemize}
				\item [$(a)$] If $r \in [2,\frac{2n}{3}]$ then
				\begin{eqnarray*}
				&&E(D_n^s[2,n-2])>E(D_n^s[6,n-6]) > \dots > E(D_n^s[\frac{2n}{3},\frac{n}{3}]) \\
				&>& E(D_n^s[\frac{2n-12}{3},\frac{n}{3}])>E(D_n^s[\frac{2n-24}{3},\frac{n}{3}])>\dots > E(D_n^s[2,\frac{n}{3}])\\
				&>& E(D_n^s[2,\frac{n-6}{3}]) > E(D_n^s[2,\frac{n-12}{3}]) >\dots > E(D_n^s[2,3]).
				\end{eqnarray*}
				\item [$(b)$] If $r \in [\frac{2n}{3},n-2]$ and $n-3\equiv 2(\bmod~4)$ then
				\begin{eqnarray*}
				&&E(D_n^s[n-3,3])>E(D_n^s[n-7,7]) > \dots > E(D_n^s[\frac{2n}{3},\frac{n}{3}])\\
				&>& E(D_n^s[\frac{2n}{3},\frac{n-6}{3}]) > E(D_n^s[\frac{2n}{3},\frac{n-12}{3}]) >\dots > E(D_n^s[\frac{2n}{3},3])\\
				&>& E(D_n^s[\frac{2n-12}{3},3]) E(D_n^s[\frac{2n-24}{3},3]) > \dots > E(D_n^s[2,3]).
				\end{eqnarray*}	
				\item [$(c)$] If $r \in [\frac{2n}{3},n-2]$ and $n-3\equiv 0(\bmod~4)$ then
				\begin{eqnarray*}
				&&E(D_n^s[n-5,5])>E(D_n^s[n-9,9]) > \dots > E(D_n^s[\frac{2n}{3},\frac{n}{3}])\\
				&>& E(D_n^s[\frac{2n-12}{3},\frac{n}{3}])>E(D_n^s[\frac{2n-24}{3},\frac{n}{3}])>\dots > E(D_n^s[2,\frac{n}{3}])\\
				&>& E(D_n^s[2,\frac{n-6}{3}]) > E(D_n^s[2,\frac{n-12}{3}]) >\dots > E(D_n^s[2,3]).
				\end{eqnarray*}	
				\item [$(d)$] If $n-3\equiv 2(\bmod~4)$ then
				\begin{eqnarray*}
				&&E(D_n^s[\boldsymbol{n-3},\boldsymbol{3}])>	E(D_n^s[\boldsymbol{n-7},\boldsymbol{7}]) > \dots > 	E(D_n^s[\boldsymbol{2},\boldsymbol{n-2}])\\
				&>& E(D_n^s[\boldsymbol{2},\boldsymbol{n-4}]) > E(D_n^s[\boldsymbol{2},\boldsymbol{n-6}])>\dots >E(D_n^s[\boldsymbol{2},\boldsymbol{3}]) >E(D_n^s[\boldsymbol{2},\boldsymbol{2}]).
				\end{eqnarray*}
				\item [$(e)$] If $n-3\equiv 0(\bmod~4)$ then
				\begin{eqnarray*}
				&&E(D_n^s[\boldsymbol{n-5},\boldsymbol{5}])>	E(D_n^s[\boldsymbol{n-9},\boldsymbol{9}]) > \dots > 	E(D_n^s[\boldsymbol{2},\boldsymbol{n-2}])\\
					&>& E(D_n^s[\boldsymbol{2},\boldsymbol{n-4}]) > E(D_n^s[\boldsymbol{2},\boldsymbol{n-6}])>\dots >E(D_n^s[\boldsymbol{2},\boldsymbol{3}]) >E(D_n^s[\boldsymbol{2},\boldsymbol{2}]).
				\end{eqnarray*}
			\end{itemize}
			\item[$(ii)$] Let $r \equiv 0(\bmod~4)$.
			\begin{itemize}
				\item [$(a)$] If $n-3\equiv 2(\bmod~4)$ then
\begin{eqnarray*}
&&E(D_n^s[n-5,5])>E(D_n^s[n-9,9]) > \dots > E(D_n^s[4,n-4])\\
&>& E(D_n^s[4,n-6]) > E(D_n^s[4,n-8]) > E(D_n^s[4,3]) > E(D_n^s[2,3]).
\end{eqnarray*}
\item [$(b)$] If $n-3\equiv 0(\bmod~4)$ then
\begin{eqnarray*}
&&E(D_n^s[n-3,5])>E(D_n^s[n-7,9]) > \dots > E(D_n^s[4,n-4])\\
&>& E(D_n^s[4,n-6]) > E(D_n^s[4,n-8]) > E(D_n^s[4,3]) > E(D_n^s[2,3]).
\end{eqnarray*}
\item [$(c)$] If $r \in [2,\frac{2n}{3}]$ then
\begin{eqnarray*}
&&E(D_n^s[\boldsymbol{4},\boldsymbol{n-4}])>E(D_n^s[\boldsymbol{8},\boldsymbol{n-8}]) > \dots > E(D_n^s[\boldsymbol{\frac{2n-6}{3}},\boldsymbol{\frac{n+6}{3}}])\\
&>& E(D_n^s[\boldsymbol{\frac{2n-18}{3}},\boldsymbol{\frac{n+6}{3}}])>E(D_n^s[\boldsymbol{\frac{2n-30}{3}},\boldsymbol{\frac{n+6}{3}}]) >\dots > E(D_n^s[\boldsymbol{4},\boldsymbol{\frac{n+6}{3}}])\\
&>& E(D_n^s[\boldsymbol{4},\boldsymbol{\frac{n}{3}}]) > E(D_n^s[\boldsymbol{4},\boldsymbol{\frac{n-6}{3}}]) >\dots > E(D_n^s[\boldsymbol{4},\boldsymbol{3}]) > E(D_n^s[\boldsymbol{2},\boldsymbol{3}]).
\end{eqnarray*}
\item [$(d)$] If $r \in [\frac{2n}{3},n-2]$ and $n-3\equiv 2(\bmod~4)$ then
\begin{eqnarray*}
&&E(D_n^s[\boldsymbol{n-5},\boldsymbol{5}])>E(D_n^s[\boldsymbol{n-9},\boldsymbol{9}]) > \dots > E(D_n^s[\boldsymbol{\frac{2n+6}{3}},\boldsymbol{\frac{n-6}{3}}])\\
				&>& E(D_n^s[\boldsymbol{\frac{2n-6}{3}},\boldsymbol{\frac{n-6}{3}}])> E(D_n^s[\boldsymbol{\frac{2n-18}{3}},\boldsymbol{\frac{n-6}{3}}])>\dots> E(D_n^s[\boldsymbol{4},\boldsymbol{\frac{n-6}{3}}])\\
				&>& E(D_n^s[\boldsymbol{4},\boldsymbol{\frac{n-12}{3}}])> E(D_n^s[\boldsymbol{4},\boldsymbol{\frac{n-18}{3}}])>\dots > E(D_n^s[\boldsymbol{4},\boldsymbol{3}]) > E(D_n^s[\boldsymbol{2},\boldsymbol{3}]) .
				\end{eqnarray*}	
				\item [$(e)$] If $r \in [\frac{2n}{3},n-2]$ and $n-3\equiv 0(\bmod~4)$ then
				\begin{eqnarray*}
				&&E(D_n^s[\boldsymbol{n-3},\boldsymbol{3}])>E(D_n^s[\boldsymbol{n-7},\boldsymbol{7}]) > \dots > E(D_n^s[\boldsymbol{\frac{2n+6}{3}},\boldsymbol{\frac{n-6}{3}}])\\
				&>& E(D_n^s[\boldsymbol{\frac{2n+6}{3}},\boldsymbol{\frac{n-12}{3}}])> E(D_n^s[\boldsymbol{\frac{2n+6}{3}},\boldsymbol{\frac{n-18}{3}}]) > \dots > E(D_n^s[\boldsymbol{\frac{2n+6}{3}},\boldsymbol{3}])\\
				&>& E(D_n^s[\boldsymbol{\frac{2n-6}{3}},\boldsymbol{3}])>E(D_n^s[\boldsymbol{\frac{2n-18}{3}},\boldsymbol{3}])>\dots > E(D_n^s[\boldsymbol{4},\boldsymbol{3}]) > E(D_n^s[\boldsymbol{2},\boldsymbol{3}]).
				\end{eqnarray*}	
			\end{itemize}
		\end{itemize}
	\item [$(2)$] If $n\equiv 1(\bmod~3)$ then we have the following energy ordering:
	\begin{itemize}
		\item [$(i)$] Let $r \equiv 2(\bmod~4)$.
		\begin{itemize}
			\item [$(a)$] If $r \in [2,\frac{2n}{3}]$ then
			\begin{eqnarray*}
			&&E(D_n^s[2,n-2])>E(D_n^s[6,n-6]) > \dots > E(D_n^s[\frac{2n-8}{3},\frac{n+8}{3}])\\
			&>& E(D_n^s[\frac{2n-20}{3},\frac{n+8}{3}]) > E(D_n^s[\frac{2n-32}{3},\frac{n+8}{3}]) >\dots > E(D_n^s[2,\frac{n+8}{3}])\\
			&>& E(D_n^s[2,\frac{n+2}{3}]) > E(D_n^s[2,\frac{n-4}{3}]) >\dots > E(D_n^s[2,3]).
			\end{eqnarray*}
			\item [$(b)$] If $r \in [\frac{2n}{3},n-2]$ and $n-3\equiv 2(\bmod~4)$ then
			\begin{eqnarray*}
			&&E(D_n^s[n-3,3])>E(D_n^s[n-7,7]) > \dots > E(D_n^s[\frac{2n+4}{3},\frac{n-4}{3}])\\
			&>& E(D_n^s[\frac{2n-8}{3},\frac{n-4}{3}]) > E(D_n^s[\frac{2n-20}{3},\frac{n-4}{3}])>\dots> E(D_n^s[2,\frac{n-4}{3}])\\
			&>& E(D_n^s[2,\frac{n-10}{3}])> E(D_n^s[2,\frac{n-16}{3}])>\dots> E(D_n^s[2,3]) .
			\end{eqnarray*}	
			\item [$(c)$] If $r \in [\frac{2n}{3},n-2]$ and $n-3\equiv 0(\bmod~4)$ then
			\begin{eqnarray*}
			&&E(D_n^s[n-5,5])>E(D_n^s[n-9,9]) > \dots > E(D_n^s[\frac{2n+4}{3},\frac{n-4}{3}])\\
			&>& E(D_n^s[\frac{2n+4}{3},\frac{n-10}{3}]) > E(D_n^s[\frac{2n+4}{3},\frac{n-16}{3}])> \dots > E(D_n^s[\frac{2n+4}{3},3])\\
			&>& E(D_n^s[\frac{2n-8}{3},3]) > E(D_n^s[\frac{2n-20}{3},3]) >\dots> E(D_n^s[2,3]) .
			\end{eqnarray*}	
				\item [$(d)$] If $n-3\equiv 2(\bmod~4)$ then
				\begin{eqnarray*}
				&&E(D_n^s[\boldsymbol{n-3},\boldsymbol{3}])>	E(D_n^s[\boldsymbol{n-7},\boldsymbol{7}]) > \dots > 	E(D_n^s[\boldsymbol{2},\boldsymbol{n-2}])\\
				&>&E(D_n^s[\boldsymbol{2},\boldsymbol{n-4}]) > E(D_n^s[\boldsymbol{2},\boldsymbol{n-6}]) >\dots > E(D_n^s[\boldsymbol{2},\boldsymbol{3}]).
				\end{eqnarray*}
				\item [$(e)$] If $n-3\equiv 0(\bmod~4)$ then
				\begin{eqnarray*}
				&&E(D_n^s[\boldsymbol{n-5},\boldsymbol{5}])>	E(D_n^s[\boldsymbol{n-9},\boldsymbol{9}]) > \dots > 	E(D_n^s[\boldsymbol{2},\boldsymbol{n-2}])\\
				&>&E(D_n^s[\boldsymbol{2},\boldsymbol{n-4}]) > E(D_n^s[\boldsymbol{2},\boldsymbol{n-6}]) >\dots > E(D_n^s[\boldsymbol{2},\boldsymbol{3}]).
				\end{eqnarray*}
		\end{itemize}
		\item[$(ii)$] Let $r \equiv 0(\bmod~4)$.
\begin{itemize}
\item [$(a)$] If $n-3\equiv 2(\bmod~4)$ then
\begin{eqnarray*}
&&E(D_n^s[n-5,5])>E(D_n^s[n-9,9]) > \dots > E(D_n^s[4,n-4])\\
&>& E(D_n^s[4,n-6]) > E(D_n^s[4,n-8]) > E(D_n^s[4,3]) > E(D_n^s[2,3]).
\end{eqnarray*}
\item [$(b)$] If $n-3\equiv 0(\bmod~4)$ then
\begin{eqnarray*}
&&E(D_n^s[n-3,5])>E(D_n^s[n-7,9]) > \dots > E(D_n^s[4,n-4])\\
&>& E(D_n^s[4,n-6]) > E(D_n^s[4,n-8]) > E(D_n^s[4,3]) > E(D_n^s[2,3]).
\end{eqnarray*}
\item [$(c)$] If $r \in [2,\frac{2n}{3}]$ then
\begin{eqnarray*}
&&E(D_n^s[\boldsymbol{4},\boldsymbol{n-4}])>E(D_n^s[\boldsymbol{8},\boldsymbol{n-8}]) > \dots > E(D_n^s[\boldsymbol{\frac{2n-2}{3}},\boldsymbol{\frac{n+2}{3}}])\\
&>& E(D_n^s[\boldsymbol{\frac{2n-14}{3}},\boldsymbol{\frac{n+2}{3}}])> E(D_n^s[\boldsymbol{\frac{2n-26}{3}},\boldsymbol{\frac{n+2}{3}}]) > \dots> E(D_n^s[\boldsymbol{4},\boldsymbol{\frac{n+2}{3}}]) \\
&>& E(D_n^s[\boldsymbol{4},\boldsymbol{\frac{n-4}{3}}]) > E(D_n^s[\boldsymbol{4},\boldsymbol{\frac{n-10}{3}}]) >\dots > E(D_n^s[\boldsymbol{4},\boldsymbol{3}]) > E(D_n^s[\boldsymbol{2},\boldsymbol{3}]).
\end{eqnarray*}
\item [$(d)$] If $r \in [\frac{2n}{3},n-2]$ and $n-3\equiv 2(\bmod~4)$ then
\begin{eqnarray*}
&&E(D_n^s[\boldsymbol{n-5},\boldsymbol{5}])>E(D_n^s[\boldsymbol{n-9},\boldsymbol{9}]) > \dots > E(D_n^s[\boldsymbol{\frac{2n+10}{3}},\boldsymbol{\frac{n-10}{3}}])\\
&>&E(D_n^s[\boldsymbol{\frac{2n-2}{3}},\boldsymbol{\frac{n-10}{3}}])> E(D_n^s[\boldsymbol{\frac{2n-14}{3}},\boldsymbol{\frac{n-10}{3}}]) >\dots > E(D_n^s[\boldsymbol{4},\boldsymbol{\frac{n-10}{3}}])\\
&>& E(D_n^s[\boldsymbol{4},\boldsymbol{\frac{n-16}{3}}]) > E(D_n^s[\boldsymbol{4},\boldsymbol{\frac{n-22}{3}}]) >\dots > E(D_n^s[\boldsymbol{4},\boldsymbol{3}]) > E(D_n^s[\boldsymbol{2},\boldsymbol{3}]).
\end{eqnarray*}	
\item [$(e)$] If $r \in [\frac{2n}{3},n-2]$ and $n-3\equiv 0(\bmod~4)$ then
\begin{eqnarray*}
&&E(D_n^s[\boldsymbol{n-3},\boldsymbol{3}])>E(D_n^s[\boldsymbol{n-7},\boldsymbol{7}]) > \dots > E(D_n^s[\boldsymbol{\frac{2n+10}{3}},\boldsymbol{\frac{n-10}{3}}])\\
&>& E(D_n^s[\boldsymbol{\frac{2n+10}{3}},\boldsymbol{\frac{n-16}{3}}]) > E(D_n^s[\boldsymbol{\frac{2n+10}{3}},\boldsymbol{\frac{n-22}{3}}]) > E(D_n^s[\boldsymbol{\frac{2n+10}{3}},\boldsymbol{3}])\\
&>&E(D_n^s[\boldsymbol{\frac{2n-2}{3}},\boldsymbol{3}]) > E(D_n^s[\boldsymbol{\frac{2n-14}{3}},\boldsymbol{3}]) > \dots > E(D_n^s[\boldsymbol{4},\boldsymbol{3}]) > E(D_n^s[\boldsymbol{2},\boldsymbol{3}]).
\end{eqnarray*}	
		\end{itemize}
	\end{itemize}
	\item [$(3)$] If $n\equiv 2(\bmod~3)$ then we have the following energy ordering:
	\begin{itemize}
		\item [$(i)$] Let $r \equiv 2(\bmod~4)$.
\begin{itemize}
	\item [$(a)$] If $r \in [2,\frac{2n}{3}]$ then
\begin{eqnarray*}
&&E(D_n^s[2,n-2])>E(D_n^s[6,n-6]) > \dots > E(D_n^s[\frac{2n-4}{3},\frac{n+4}{3}])\\
&>& E(D_n^s[\frac{2n-16}{3},\frac{n+4}{3}]) > E(D_n^s[\frac{2n-26}{3},\frac{n+4}{3}]) >\dots> E(D_n^s[2,\frac{n+4}{3}])\\
&>& E(D_n^s[2,\frac{n-2}{3}]) > E(D_n^s[2,\frac{n-8}{3}]) > E(D_n^s[2,3]).
\end{eqnarray*}
\item [$(b)$] If $r \in [\frac{2n}{3},n-2]$ and $n-3\equiv 2(\bmod~4)$ then
\begin{eqnarray*}
&&E(D_n^s[n-3,3])>E(D_n^s[n-7,7]) > \dots > E(D_n^s[\frac{2n+8}{3},\frac{n-8}{3}])\\
&>& E(D_n^s[\frac{2n-4}{3},\frac{n-8}{3}]) > E(D_n^s[\frac{2n-16}{3},\frac{n-8}{3}]) > \dots > E(D_n^s[2,\frac{n-8}{3}])\\
&>&E(D_n^s[2,\frac{n-14}{3}]) > E(D_n^s[2,\frac{n-20}{3}]) > \dots > E(D_n^s[2,3]) .
\end{eqnarray*}	
\item [$(c)$] If $r \in [\frac{2n}{3},n-2]$ and $n-3\equiv 0(\bmod~4)$ then
\begin{eqnarray*}
&&E(D_n^s[n-5,5])>E(D_n^s[n-9,9]) > \dots > E(D_n^s[\frac{2n+8}{3},\frac{n-8}{3}])\\
&>& E(D_n^s[\frac{2n+8}{3},\frac{n-14}{3}]) > E(D_n^s[\frac{2n+8}{3},\frac{n-20}{3}]) > \dots > E(D_n^s[\frac{2n+8}{3},3])\\
&>& E(D_n^s[\frac{2n-4}{3},3]) > E(D_n^s[\frac{2n-16}{3},3]) > \dots > E(D_n^s[2,3]).
\end{eqnarray*}
\item [$(d)$] If $n-3\equiv 2(\bmod~4)$ then
\begin{eqnarray*}
&&E(D_n^s[\boldsymbol{n-3},\boldsymbol{3}])>	E(D_n^s[\boldsymbol{n-7},\boldsymbol{7}]) > \dots > 	E(D_n^s[\boldsymbol{2},\boldsymbol{n-2}])\\
&>&E(D_n^s[\boldsymbol{2},\boldsymbol{n-4}]) > E(D_n^s[\boldsymbol{2},\boldsymbol{n-6}]) >\dots > E(D_n^s[\boldsymbol{2},\boldsymbol{3}]).
\end{eqnarray*}
\item [$(e)$] If $n-3\equiv 0(\bmod~4)$ then
\begin{eqnarray*}
&&E(D_n^s[\boldsymbol{n-5},\boldsymbol{5}])>	E(D_n^s[\boldsymbol{n-9},\boldsymbol{9}]) > \dots > 	E(D_n^s[\boldsymbol{2},\boldsymbol{n-2}])\\
&>&E(D_n^s[\boldsymbol{2},\boldsymbol{n-4}]) > E(D_n^s[\boldsymbol{2},\boldsymbol{n-6}]) >\dots > E(D_n^s[\boldsymbol{2},\boldsymbol{3}]).
\end{eqnarray*}	
\end{itemize}
\item[$(ii)$] Let $r \equiv 0(\bmod~4)$.
\begin{itemize}
\item [$(a)$] If $n-3\equiv 2(\bmod~4)$ then
\begin{eqnarray*}
&&E(D_n^s[n-5,5])>E(D_n^s[n-9,9]) > \dots > E(D_n^s[4,n-4])\\
&>& E(D_n^s[4,n-6]) > E(D_n^s[4,n-8]) > E(D_n^s[4,3]) > E(D_n^s[2,3]).
\end{eqnarray*}
\item [$(b)$] If $n-3\equiv 0(\bmod~4)$ then
\begin{eqnarray*}
&&E(D_n^s[n-3,5])>E(D_n^s[n-7,9]) > \dots > E(D_n^s[4,n-4])\\
&>& E(D_n^s[4,n-6]) > E(D_n^s[4,n-8]) > E(D_n^s[4,3]) > E(D_n^s[2,3]).
\end{eqnarray*}
\item [$(c)$] If $r \in [2,\frac{2n}{3}]$ then
\begin{eqnarray*}
&&E(D_n^s[\boldsymbol{4},\boldsymbol{n-4}])>E(D_n^s[\boldsymbol{8},\boldsymbol{n-8}]) > \dots > E(D_n^s[\boldsymbol{\frac{2n-10}{3}},\boldsymbol{\frac{n+10}{3}}])\\
&>& E(D_n^s[\boldsymbol{\frac{2n-22}{3}},\boldsymbol{\frac{n+10}{3}}]) > E(D_n^s[\boldsymbol{\frac{2n-34}{3}},\boldsymbol{\frac{n+10}{3}}]) > \dots > E(D_n^s[\boldsymbol{4},\boldsymbol{\frac{n+10}{3}}])\\
&>&E(D_n^s[\boldsymbol{4},\boldsymbol{\frac{n+4}{3}}]) > E(D_n^s[\boldsymbol{4},\boldsymbol{\frac{n-2}{3}}]) > \dots > E(D_n^s[\boldsymbol{4},\boldsymbol{3}]) > E(D_n^s[\boldsymbol{2},\boldsymbol{3}]).
\end{eqnarray*}
\item [$(d)$] If $r \in [\frac{2n}{3},n-2]$ and $n-3\equiv 2(\bmod~4)$ then
\begin{eqnarray*}
&&E(D_n^s[\boldsymbol{n-5},\boldsymbol{5}])>E(D_n^s[\boldsymbol{n-9},\boldsymbol{9}]) > \dots > E(D_n^s[\boldsymbol{\frac{2n+2}{3}},\boldsymbol{\frac{n-2}{3}}])\\
&>& E(D_n^s[\boldsymbol{\frac{2n-10}{3}},\boldsymbol{\frac{n-2}{3}}]) > E(D_n^s[\boldsymbol{\frac{2n-22}{3}},\boldsymbol{\frac{n-2}{3}}]) > E(D_n^s[\boldsymbol{4},\boldsymbol{\frac{n-2}{3}}])\\
&>& E(D_n^s[\boldsymbol{4},\boldsymbol{\frac{n-8}{3}}]) > E(D_n^s[\boldsymbol{4},\boldsymbol{\frac{n-14}{3}}]) >\dots> E(D_n^s[\boldsymbol{4},\boldsymbol{3}])>  E(D_n^s[\boldsymbol{2},\boldsymbol{3}]).
\end{eqnarray*}	
\item [$(e)$] If $r \in [\frac{2n}{3},n-2]$ and $n-3\equiv 0(\bmod~4)$ then
\begin{eqnarray*}
&&E(D_n^s[\boldsymbol{n-3},\boldsymbol{3}])>E(D_n^s[\boldsymbol{n-7},\boldsymbol{7}]) > \dots > E(D_n^s[\boldsymbol{\frac{2n+2}{3}},\boldsymbol{\frac{n-2}{3}}])\\
&>& E(D_n^s[\boldsymbol{\frac{2n+2}{3}},\boldsymbol{\frac{n-8}{3}}]) > E(D_n^s[\boldsymbol{\frac{2n+2}{3}},\boldsymbol{\frac{n-14}{3}}]) >\dots > E(D_n^s[\boldsymbol{\frac{2n+2}{3}},\boldsymbol{3}])\\
&>& E(D_n^s[\boldsymbol{\frac{2n-10}{3}},\boldsymbol{3}]) > E(D_n^s[\boldsymbol{\frac{2n-22}{3}},\boldsymbol{3}]) >\dots > E(D_n^s[\boldsymbol{4},\boldsymbol{3}]) > E(D_n^s[\boldsymbol{2},\boldsymbol{3}]).
\end{eqnarray*}	
\end{itemize}
	\end{itemize}
	\end{itemize}
\end{theorem}

\end{document}